\documentclass[11pt]{article}
\usepackage[letterpaper, margin=1in]{geometry}
\usepackage[utf8]{inputenc}

\usepackage{amsmath, amssymb, amsthm, authblk, bbm, mathtools, xcolor}
\usepackage{enumerate}
\usepackage{caption}
\usepackage{subcaption}
\usepackage[sort&compress,numbers]{natbib}
\usepackage[linktocpage=true]{hyperref}
\hypersetup{colorlinks=true, linkcolor=blue}

\mathtoolsset{showonlyrefs}

\numberwithin{equation}{section}
\newtheorem{theorem}{Theorem}[section]

\newtheorem{assumption}[theorem]{Assumption}

\newtheorem{lemma}[theorem]{Lemma}
\newtheorem{remark}[theorem]{Remark}
\newtheorem{claim}[theorem]{Claim}

\let\plainqed\qedsymbol
\newcommand{\claimqed}{$\lrcorner$}
\newenvironment{claimproof}{\begin{proof}\renewcommand{\qedsymbol}{\claimqed}}{\end{proof}\renewcommand{\qedsymbol}{\plainqed}}

\newcommand{\keywords}[1]{\textbf{\textit{Keywords ---}} #1}
\newcommand*\diff{\mathop{}\!\mathrm{d}}

\renewcommand{\P}{\mathbb{P}}
\newcommand{\E}{\mathbb{E}}
\newcommand{\R}{\mathbb{R}}
\newcommand{\N}{\mathbb{N}}
\newcommand{\bmu}{\boldsymbol{\mu}}

\begin{document}

\title{Distributed Rate Scaling in Large-Scale Service Systems}

\author[1]{Daan Rutten\thanks{Email: \href{mailto:drutten@gatech.edu}{drutten@gatech.edu}}\textsuperscript{,}}
\author[2]{Martin Zubeldia}
\author[1]{Debankur Mukherjee}

\affil[1]{Georgia Institute of Technology}
\affil[2]{University of Minnesota}

\maketitle

\begin{abstract}
We consider a large-scale parallel-server system, where each server independently adjusts its processing speed in a decentralized manner. The objective is to minimize the overall cost, which comprises the average cost of maintaining the servers' processing speeds and a non-decreasing function of the tasks' sojourn times. The problem is compounded by the lack of knowledge of the task arrival rate and the absence of a centralized control or communication among the servers.
We draw on ideas from stochastic approximation and present a novel rate scaling algorithm that ensures convergence of all server processing speeds to the globally asymptotically optimum value as the system size increases. 
Apart from the algorithm design, a key contribution of our approach lies in demonstrating how concepts from the stochastic approximation literature can be leveraged to effectively tackle learning problems in large-scale, distributed systems.
En route, we also analyze the performance of a fully heterogeneous parallel-server system, where each server has a distinct processing speed, which might be of independent interest.
\end{abstract}


\section{Introduction} 

The escalating power consumption of data centers has become a pressing concern in recent years. If left unchecked, the electricity demand of data centers is predicted to grow up to 8\% of the total U.S.~electricity consumption by 2030~\cite{jones2018stop}. 
As a result, data center providers are continuously working towards maximizing energy efficiency in their servers by pushing the hardware's capabilities to their peak potential. 
However, to achieve significant additional improvements, there is a renewed interest in enhanced algorithm design in this domain~\cite{Shebabi2016}.
One popular approach to improve power efficiency has been dynamic speed scaling~\cite{bansal2009speed, wierman2009power, albers2014speed, ata2006dynamic}, in which  processing speeds of servers are dynamically scaled to strike a delicate balance between the system's power consumption and the user-perceived quality of service (such as delay).
The main contribution of the current paper is to propose a novel algorithm for speed scaling in the context of large-scale implementations and establish its asymptotic optimality.

In modern chips, the power consumption exhibits a polynomial growth pattern in relation to the processing speed (typically cubic), while the processing time is clearly inversely related. 
These two features are the building blocks of the global system cost, which is a carefully crafted combination of the power consumption as well as the user-perceived sojourn time. 
The problem can be framed as an online optimization problem, requiring servers to determine their service rate in order to minimize a specific global objective function; see~\eqref{eq:costfunction}. 
This problem is particularly relevant for large-scale data centers, which house tens of thousands of servers. Any centralized policy in such facilities is prohibitive due to the scale of operations and their implementation complexity. 
Therefore, any algorithm for this optimization problem must be \textit{scalable}.
Even more importantly, the global arrival rate of tasks is often unknown, especially to individual servers. Therefore, any algorithm for this optimization problem must \textit{dynamically learn and adapt} to the environment. 
Motivated by the above, in this paper, we present a distributed, online solution to the aforementioned optimization problem. 
The proposed algorithm enables each server to autonomously converge towards its globally optimal service rate without requiring any explicit inter-server communication or knowledge of the incoming traffic intensity.

\subsection{Previous work}

There are several lines of works in the literature that are relevant to the problem under consideration. 
In the absence of stochastic arrivals, speed scaling in the single-server case has drawn considerable attention.  
Here, the predominant tool has been competitive analysis. Bansal et al.~\cite{bansal2009speed} considered the shortest remaining processing time (SRPT) service disciple, Wierman et al.~\cite{wierman2009power} studied processor sharing systems and Albers et al.~\cite{albers2014speed} analyzed parallel processors with deadline constraints.
Moreover, Ata and Shneorson~\cite{ata2006dynamic} considered the speed scaling problem when the inter-arrival times and processing times are exponentially distributed.
The setup in the above works is the closest to our work in spirit, but it lacks the fundamental challenge of tackling dynamically evolving heterogeneous systems, which are inherently high-dimensional, under the effects of a load balancing algorithm.

In the case of parallel server systems, an alternative approach for reducing system energy consumption is to dynamically scale the number of `active servers', or turning servers on and off, to avoid unnecessary idleness by the servers.
There has been a large body of works in designing algorithms true to this approach~\cite{Gandhi2010, gandhi2013exact, LWAT13,  maccio2015optimal, MDBL17, MS19, RM23}.
It is worth mentioning that this line of works predominantly consider homogeneous servers.
Here, the goal is to have the minimum number of servers active such that the performance, in terms of delay, is not degraded by congestion. 
The main challenge in this scenario is the fact that servers cannot be turned on instantaneously, so systems have to be careful to ensure that the active capacity is enough at all times. 

Finally, there are several works on parallel server systems with heterogeneous server rates. In the scheduling literature, the literature has analyzed scheduling tasks across different machines while also adapting to their service rates \cite{albers2014speed, Nair2020, LearningEnergyScheduling}. There are a few papers closer to our setup, which consider a joint dispatching and speed scaling at a fluid level, where the fraction of tasks and the speed of the servers are chosen to minimize a tractable cost function \cite{scaling22,HUANG201924, DynSpeedScaling}. The authors develop decentralized algorithms which minimize the cost, either by writing the optimization as a Lyapunov function plus a penalty and then moving along the drift (\cite{scaling22}), by employing a Lagrange method and binary search approach (\cite{HUANG201924}), or by sharing local information with neighbors to solve the convex optimization problem in a decentralized way (\cite{DynSpeedScaling}). While these papers consider dynamic rate scaling, the fact that they customize and control the load balancing policy greatly simplifies both the objective function and the analysis.

\subsection{Our contributions}
In this paper, we design and analyze a simple, distributed algorithm that runs decentralized at each server and updates the service rate based only on the idle times of the local server. 
In the relevant asymptotic regime, we prove that the cost of the server rates under our algorithm converge to the globally optimal cost. 
More specifically, we make contributions on three fronts:\\

\noindent
\textbf{(a)} \emph{Heterogeneous service rates}. We mentioned before that the objective function involves the (steady-state expected) sojourn time of tasks. Therefore, to evaluate the system performance for any speed profile of the system, or to even formulate the optimization problem, it is crucial to better understand the sojourn time of tasks in fully heterogeneous large-scale systems, i.e., where all service rates are distinct. 
To the best of our knowledge, in this setup, expressions for the average queue length or task sojourn times under policies like Join-Idle-Queue (JIQ) are unknown till date. 
More fundamentally, we have reasons to believe that simple expressions for these statistics may not even exist.
To tackle this problem, we prove that the steady state expected idle time of different servers becomes equal in the many-server limit, even for a heterogeneous system. 
To obtain this result, we use an auxiliary measure-valued process to bridge the gap between the intractable system state and a more manageable deterministic approximation. 
Although we do not have a closed form expression, the characterization of the idle time allows us to describe the trajectories of the service rates and compute the expected sojourn time in the limit. 
En route, we show a concentration result on the fraction of busy servers in the system.\\

\noindent
\textbf{(b)}\ \emph{Local gradient descent of global cost}. A popular approach in continuous, convex optimization is \emph{gradient descent} due to its simple formulation and convergence guarantees. Here, the decision variable is moved incrementally in the negative gradient direction, which, for sufficiently small steps, guarantees that the objective decreases in every step. 
At this point, the literature has explored numerous variants of gradient descent. 
For example, of particular relevance to the current scenario, when only noisy estimates of the gradient are available \cite{RobbinsMonro,gardner1984learning,Polyak,Nemirovski,yuan2016convergence}. 
However, a gradient descent-type approach cannot directly be applied to the current setup for several reasons. 
First, as discussed before, a closed form expression for the expected sojourn time in a heterogeneous system does not exist. As a result, the gradient of the global cost function cannot be evaluated by an algorithm. 
Moreover, we cannot assume that such a quantity would be a convex and well-behaved (such as, differentiable) function of the service rates for a policy such as Join-Idle-Queue (JIQ). 
Second, and more importantly, even if an expression for the sojourn time would be available, it would depend on the service rates of \emph{all servers} in the system and hence cannot be computed at a single server without communication. 
Similarly, any \emph{unbiased, noisy estimates} of the sojourn times cannot be obtained either. 
For example, the number of tasks a server receives in a particular time interval does capture certain global statistics of the system but is ultimately biased by the service rate of the server itself, which violates crucial assumptions of this method. This is further explained in Section \ref{sec:idleAnalysis}.

We overcome the above difficulties by leveraging the fact that the expected idle times of the servers become equal in the many-server limit, even under heterogeneity.
The proof of this fact constitutes the bulk of the technical part of the paper.
Using this, we obtain an asymptotically tight (convex) lower bound on the cost function, and show that this lower bound is minimized by a homogeneous system (Theorem~\ref{thm:optimum}).\\

\noindent
\textbf{(c)}\ \emph{Distributed optimization.} 
Although there are numerous approaches to distributed optimization problems in the literature, all of them allow some form of (limited) communication between the agents (servers in our case) \cite{tsitsiklis1984,distributed86,Nedic18,yang2019survey}. 
In contrast, we do not allow any explicit communication between servers. This renders any possibility to arrive at a global optimum solution hopeless in traditional distributed optimization settings. 
However, our setup differs in one important way: the servers receive tasks via the JIQ load balancing policy that \emph{does} have some global `signal' about the system, i.e., the policy assigns tasks to idle servers, whenever there are any. 
The challenge is, therefore, to carefully use these sparse and implicit hints from the load balancing policy to change a server's service rate. 
For example, under the JIQ policy, if a server is idle and receives a new task very fast, then one of two possibilities could be happening:
(i) the system is overloaded and the server should increase its service rate; or
(ii) the service rate of that specific server is high enough but the service rate of other servers is too low. 
The algorithm should be able to distinguish these two possibilities.

Our algorithm is inspired by stochastic approximation algorithms with constant step sizes that provide performance guarantees in the asymptotic regime where the step size tends to zero. 
By exploiting the fact that the expected idle time of servers becomes equal in the many-server limit, we show that the service rate of each server converges to the globally optimal service rate under our algorithm.


\section{Model and main results}\label{sec:main}
We consider a system consisting of $n \in \N$ parallel servers with unit buffers, where server $v \in [n]$ has processing rate $\mu_v^{n}>0$. Tasks arrive as a Poisson process of rate $\lambda n$, and have i.i.d.~unit-mean exponential processing time requirements. Upon arrival, tasks are either routed to an idle server chosen uniformly at random, if there are any, or dropped otherwise. The goal is for the servers to run at processing rates that attain the optimal cost:
\begin{equation}\label{eq:costfunction}
    \inf\limits_{\bmu^n \in \mathbb{R}_+^n } \left\{\lambda g\left( \frac{1}{\E[ S^{n}(\infty) ]} \right) + \frac{1}{n} \sum_{v \in [n]} h(\mu_v^{n}) \right\},
\end{equation}
where $\E[ S^{n}(\infty) ]$ is the expected sojourn time of a typical task in steady state (which is a function of the vector of processing rates $\bmu^n$), $g:\mathbb{R}_+\to\mathbb{R}_+$ is a decreasing function which represents the cost of large sojourn times of tasks (hence this cost is weighed by the arrival rate) and $h:\mathbb{R}_+\to\mathbb{R}_+$ is an increasing function which represents the cost of maintaining a specific processing speed. We assume that these functions satisfy the following assumption.

\begin{assumption}\label{ass:opt}
We assume that:
\begin{enumerate}[{\normalfont (i)}]
    \item $g$ and $h$ are twice continuously differentiable;
    \item There exists $\sigma_g, \sigma_h > 0$, such that $g''(x) \geq \sigma_g$ and $h''(x) \geq \sigma_h$, for all $x\geq 0$;
    \item $g$ is decreasing, $h$ is increasing, and $h'(x) / x$ is non-decreasing in $x$;
    \item $h'(0)=0$;
    \item $\lambda g'(\lambda) + h'(\lambda) < 0$;
    \item There exist $\mu_+ > \mu_- > 0$ such that
     $g'(\mu_+) + \frac{h'(\mu_+)}{\mu_+} \geq 0$ and $g'(\mu_-) + h'(\mu_-) \left(\lambda + \frac{1}{\mu_-}\right) \leq 0.$ 
\end{enumerate}
\end{assumption}


\begin{remark}\normalfont
An example of functions $g$ and $h$ that satisfy Assumption \ref{ass:opt} is
\[ g(x) = \frac{1}{1+x} \qquad \text{and} \qquad h(x)=\beta x^3, \]
where $\beta>0$ is a small enough constant.The use of a polynomial function for the energy costs, and in particular of a cubic one, is well supported in the literature \cite{bansal2009speed,wierman2009power}.
Indeed, (i) and (ii) are clearly satisfied, and since
$g'(x) = -(1+x)^{-2}$ and $h'(x)= 3\beta x^2$,
(iii) and (iv) are also satisfied for all $\beta>0$. Moreover, we have
 $\lambda g'(\lambda) + h'(\lambda) = -\lambda(1+\lambda)^{-2} + 3\beta \lambda^2$, 
which is negative for all $\beta$ small enough, so (v) is satisfied when choosing $\beta$ small enough. Finally, we have
$g'(\mu_+) + h'(\mu_+)/\mu_+ = -(1+\mu_+)^{-2} + 3\beta \mu_+$, 
which is non-negative for $\mu_+$ sufficiently large for any $\beta>0$, and
$g'(\mu_-) + h'(\mu_-) \left(\lambda + 1/\mu_-\right) = -(1+\mu_-)^{-2} + 3\beta \mu_- \left(\lambda \mu_- + 1\right)$, 
which is non-positive for $\mu_-$ sufficiently small for any $\beta>0$. Thus, (vi) is also satisfied.
\end{remark}

In Assumption \ref{ass:opt}, (i) are (ii) are standard convexity assumptions. However, since the expected sojourn time of jobs is an unknown function of the service rate vector, our objective function may be non-convex, and possibly even discontinuous. Therefore, service rate vectors which achieve the infimum cost may not exists or there may be multiple ones. As a tractable upper bound for the infimum cost, we consider homogeneous systems with rate $\mu\geq 0$, for which the cost function reduces to the convex function $\lambda g(\mu) + h(\mu)$, for all $n$. Constrained to homogeneous systems, it is easily checked that Assumption \ref{ass:opt} (i), (ii), and (iii) imply that the infimum cost is indeed attained at
\begin{equation}
    \mu^* := \underset{\mu \geq 0}{\arg\min} \big\{ \lambda g(\mu) + h(\mu) \big\},
\end{equation}
with $\mu^*>\lambda$ by Assumption \ref{ass:opt} (v). Therefore, we have
\begin{equation}
    \inf\limits_{\bmu^n \in \mathbb{R}_+^n } \left\{\lambda g\left( \frac{1}{\E[ S^{n}(\infty) ]} \right) + \frac{1}{n} \sum_{v \in [n]} h(\mu_v^{n}) \right\} \leq \lambda g(\mu^*) + h(\mu^*).
\end{equation}
Our first result states that the gap in the above inequality vanishes, as $n\to\infty$, and that the system is stable when all servers run at rate $\mu^*$. In particular, this means that the homogeneous rate vector where all entries are equal to $\mu^*$ asymptotically minimizes the cost of Equation \eqref{eq:costfunction}, in the limit as $n\to\infty$ (i.e., in the many-server limit).


\begin{theorem}\label{thm:optimum}
Under Assumption~\ref{ass:opt}, we have $\mu^*>\lambda$, and
\[ \lim_{n\to\infty} \inf\limits_{\bmu^n \in \mathbb{R}_+^n } \left\{\lambda g\left( \frac{1}{\E[ S^{n}(\infty) ]} \right) + \frac{1}{n} \sum_{v \in [n]} h(\mu_v^{n}) \right\} = \lambda g(\mu^*) + h(\mu^*). \]
\end{theorem}
To establish this result, we show that, for a given average rate, a homogeneous rate vector minimizes the cost when $n\to\infty$. While the convexity of $h$ easily implies that homogeneous processing rates minimize the second term in Equation \eqref{eq:costfunction} for any given average rate, the fact that this also minimizes the first term crucially relies on the fact that the expected idle times of servers become equal as $n\to\infty$ (cf. Theorem \ref{thm:idle}). The proof is given in Section~\ref{sec:proof_optimum}.

\subsection{Adaptive rate scaling: ODE characterization of  sample paths}
We now present our rate scaling algorithm, which dynamically adjusts the individual servers' processing rates. Since servers are oblivious to the global arrival rate $\lambda$ and cannot communicate with each other, we design an adaptive algorithm which learns the asymptotically optimal processing rates in a completely distributed way. Our algorithm is designed around the key insight that: \emph{the average idle time of each server becomes equal as $n \to \infty$ and therefore servers may use their idle times as `signals' (without any explicit exchange of information among servers) to gauge how their service rate compares to the average service rate in the system, and adjust accordingly.} Under our algorithm, the service rate of each server is updated as follows:
\begin{equation}\label{eq:algo}
\begin{multlined}
    \mu_v^{n, m}(t)
    = \mu_v^{n, m}(0) - \frac{1}{m} \int\limits_0^{t} \left[g'(\mu_v^{n,m}(s)) + h'(\mu_v^{n,m}(s) ) \left( I_v^{n,m}(s) + \frac{1}{\mu_v^{n,m}(s)} \right)\right] \diff s,
\end{multlined}
\end{equation}
where $m>0$ is a tunable hyperparameter, and $I_v^{n,m}(t)$ denotes the \emph{idle time} of server $v \in [n]$ at time $t$. That is, if server $v$ is idle at time $t$, then $I_v^{n,m}(t)$ is the length of time since the last service completion, and if server $v$ is busy at time $t$, then $I_v^{n,m}(t)$ is the length of its last idle period. Moreover, since Assumption \ref{ass:opt} (i) and (iii) imply that the limit $\lim_{x\to 0} h'(x)/x$ exists and is finite, we use the convention that $h'(0)/0=\lim_{x\to 0} h'(x)/x$ for the update rule to be well-defined at $\mu_v^{n,m}=0$.

\begin{remark}
    If the buffers are larger than unit size, tasks can be queued and thus processed one after another without the server being idle in between. In this case, we consider the length of the idle period to be zero. This signals that the server might be overloaded, and that it should increase its service rate if it is not too large already.
\end{remark}

\begin{remark} \label{rem:bounds} \normalfont
Assumption \ref{ass:opt} (ii), (iii), and (vi), and the convention $h'(0)/0=\lim_{x\to 0} h'(x)/x$, imply that $g'(0)+h'(0)I_v^{n,m}+h'(0)/0 \leq 0$ regardless of the value of $I_v^{n,m}$. Therefore, the integrated function in Equation \eqref{eq:algo} is non-positive when $\mu_v^{n,m}=0$ and thus $\mu_v^{n,m}(t) \geq 0$ for all $t \geq 0$ and $v \in [n]$. Moreover, Assumption \ref{ass:opt} (vi) implies that $g(\mu_+)+h'(\mu_+)(I_v^{n,m}+1/\mu_+) \geq g(\mu_+)+h'(\mu_+)/\mu_+ \geq 0$. Therefore, the integrated function in Equation \eqref{eq:algo} is non-negative when $\mu_v^{n,m}=\mu_+$ and thus, if $\mu_v^{n,m}(0)\leq \mu_+$, then $\mu_v^{n,m}(t)\leq \mu_+$ for all $t\geq 0$. We combine these two facts to conclude that, if $\mu^{n,m}(0) \in [0,\mu_+]^n$, then $\mu^{n,m}(t) \in [0,\mu_+]^n$ for all $t\geq 0$.
\end{remark}

Note that Equation \eqref{eq:algo} resembles a continuous-time version of a stochastic approximation algorithm with constant step size $1/m$. Thus, the parameter $m$ determines the learning rate of the update rule. In order to analyze the algorithm, we consider the asymptotic regime where $m \to \infty$. As $m$ gets larger, the service rates are updated at a slower pace, but with less randomness. Thus, in the spirit of stochastic approximation, if we accelerate time by a factor of $m$, we obtain a deterministic limiting trajectory with a constant learning rate. This is formalized in the following theorem.

\begin{theorem}\label{thm:process_mconv}
If $\boldsymbol{\mu}^{n, m}(0) \to \boldsymbol{\mu}^n(0)$ weakly as $m \to \infty$, then the sequence of stochastic trajectories $\big\{ (\bmu^{n,m}(mt))_{t\geq 0} \big\}_{m \in \N}$ is relatively compact with respect to the weak topology. Moreover, all limit points $(\bmu^n(t))_{t\geq 0}$ satisfy
\begin{equation}\label{eq:limit}
    \mu_v^n(t) = \mu_v^n(0) - \int\limits_0^t \left[g'(\mu_v^n(s)) + h'(\mu_v^n(s)) \left( \E_{\boldsymbol{\mu}^n(t)}\left[ I_v^n(\infty) \right] + \frac{1}{\mu_v^n(s)} \right)\right] \diff s,
\end{equation}
where $\E_{\boldsymbol{\mu}}[ I_v^n(\infty) ]$ denotes the expected idle time of server $v \in [n]$ in steady state in a system where the service rate vector is fixed at $\boldsymbol{\mu}$.
\end{theorem}

The proof is based on time-scale separation ideas introduced by Kurtz~\cite{kurtz1992averaging}, and it is given in Section~\ref{sec:proof_process_mconv}. 

\begin{remark} \normalfont
    Since the expected idle time $\E_{\boldsymbol{\mu}^n}\left[ I_v^n(\infty) \right]$ is an unknown function of the rate vector $\boldsymbol{\mu}^n$, uniqueness of solutions to Equation \eqref{eq:limit} cannot be guaranteed. Without this uniqueness, limiting trajectories may be different (even if they all satisfy Equation \eqref{eq:limit}), and thus the convergence of the sequence of trajectories cannot be guaranteed either. Fortunately, as we will see, this does not pose a problem for our subsequent analysis, where we prove the convergence of $\bmu^n(t)$ in Equation~\eqref{eq:limit} to one of the asymptotically optimal vector of service rates, as $n\to\infty$.
\end{remark}

\begin{remark}\label{rem:mu-up-low}\normalfont
Since the instantaneous arrival rate to any idle server is at least $\lambda$, regardless of the system occupancy, it follows that $\E_{\boldsymbol{\mu}}\left[ I_v^n(\infty) \right] \leq \lambda$. Combining this with Assumption~\ref{ass:opt} (vi) we get that $g'(\mu_-) + h'(\mu_-) \big(\E_{\boldsymbol{\mu}_-}[ I_v^n(\infty) ] + 1/\mu_-\big) \leq g'(\mu_-) + h'(\mu_-) \left(\lambda + 1/\mu_-\right) \leq 0$. Therefore, the integrated function in Equation \eqref{eq:limit} is non-negative when $\mu_v^n(t)=\mu_-$ and thus, if $\mu_v^n(0) \geq \mu_-$, then $\mu_v^n(t) \geq \mu_-$ for all $t\geq 0$. We combine this with Remark~\ref{rem:bounds} to conclude that, if $\mu^n(0) \in [\mu_-,\mu_+]^n$, then $\mu^n(t) \in [\mu_-,\mu_+]^n$ for all $t\geq 0$.
\end{remark}

\subsection{Idle-time analysis of large-scale heterogeneous server systems}\label{sec:idleAnalysis}

In order to understand the limiting dynamics given by Equation~\eqref{eq:limit}, we need to get characterize $\E_{\boldsymbol{\mu}}\left[ I_v^n(\infty) \right]$, which is the expected idle time of a system where the service rates are fixed at $\mu_v$ for $v \in [n]$. This poses a major challenge, as the dynamics of a fully heterogeneous system appear to be intractable. In particular, even if the dispatching policy chooses an idle server uniformly at random regardless of its processing rate, the idle time of a server $v$ depends on the number of idle servers in the system \emph{conditioned on $v$ being idle}, thus breaking the symmetry of the dispatching policy. Although a closed-form expression is out of reach, we prove in Theorem~\ref{thm:idle} below that, in the limit as $n\to\infty$, the expected idle time for different servers becomes equal. 

\begin{theorem} \label{thm:idle}
There exists constants $c^n(\bmu^n) \in [0,1]$, $n\in \N$, which depend only on $n$ and $\bmu^n$ such that
\begin{equation}
\begin{multlined}
    \max_{v \in [n]} \left\lvert \E_{\bmu^n}\left[ I^n_v(\infty) \right] - \frac{1 - c^n(\bmu^n)}{\lambda} \right\rvert \to 0 \text{ as } n \to \infty.
\end{multlined}
\end{equation}
Moreover, if $\frac{1}{n}\sum\limits_{v\in[n]} \mu^n_v \leq \lambda$, we have $c^n(\bmu^n)=1$, and
if $\liminf\limits_{n\to\infty} \frac{1}{n} \sum\limits_{v\in[n]} \mu^n_v > \lambda$, we have
\begin{equation}
    \lim_{n\to\infty} \frac{\lambda}{\max\limits_{v\in[n]} \mu^n_v} \leq \liminf_{n\to\infty} c^n(\bmu^n) \leq \limsup_{n\to\infty} c^n(\bmu^n) \leq \lim_{n\to\infty} \frac{\lambda}{\min\limits_{v\in[n]} \mu^n_v}.
\end{equation}
\end{theorem}
The proof of Theorem~\ref{thm:idle} is technically involved and is given in Section \ref{sec:proof_subcrit_idle}. The key intermediate step is a concentration result on the fraction of busy servers.

\subsection{Asymptotic optimality of the proposed algorithm}
Although we do not have an explicit expression of the constants $c^n(\bmu^n)$ of Theorem \ref{thm:idle} in general, we do have an expression for the case when the system is homogeneous (which is the case for the asymptotically optimal solution). We build upon this result to establish the next crucial lemma, which states that the deterministic trajectories converge to an asymptotically homogeneous system.

\begin{lemma}
\label{lem:maxmin_conv}
Let $\bmu^n(t)$ be as defined in Theorem \ref{thm:process_mconv}. Then,
\[ \max_{v,v' \in [n]} \left|\mu_v^n(t) - \mu_{v'}^n(t)\right| \to 0, \text{ as } n \to \infty \text{ and } t \to \infty, \]
regardless of the order in which the two limits are taken.
\end{lemma}
The proof of Lemma~\ref{lem:maxmin_conv} consists of using the bound given in Theorem \ref{thm:idle} to bound the derivative of the difference of the rates, and show that this difference converges to zero. The proof is given in Section \ref{sec:proof_maxmin_conv}.\\

Before stating our last result, we provide an intuitive explanation as of why theorems \ref{thm:idle} and \ref{lem:maxmin_conv} imply that the rates of all servers converge to the asymptotically optimal rate $\mu^*$, as $n$ and $t$ go to infinity. First, note that Theorems \ref{thm:idle} and \ref{lem:maxmin_conv} imply that, for every $\varepsilon>0$, there exist $n_0$ and $t_0$ large enough so that $\max_{v \in [n]} \big\lvert \E_{\bmu^n}\big[ I^n_v(\infty) \big] - \big[1 - c^n(\bmu^n)\big]/\lambda \big\rvert < \varepsilon$ and $\max_{v,v' \in [n]} \left|\mu_v^n(t) - \mu_{v'}^n(t)\right|<\varepsilon$ for all $n\geq n_0$ and $t\geq t_0$. With this in mind, consider a system where the expected idle times of different servers are all equal and where the service rates are homogeneous after some time. In this hypothetical system, the update rule now equals
\begin{align*}
    \frac{\diff \mu_v^n(t)}{\diff t}
    &= - g'(\mu_v^n(t)) - h'(\mu_v^n(t)) \left( \frac{1 - \lambda / \mu_v^n(t)}{\lambda} - \frac{1}{\mu_v^n(t)} \right) 
    = - g'(\mu_v^n(t)) - \frac{h'(\mu_v^n(t))}{\lambda}.
\end{align*}
Notice that the update rule equals the negative gradient of the cost function, as one would expect in gradient descent. Therefore, we expect the service rate to converge to the optimal solution. This intuition is formalized in the theorem below.

\begin{theorem} \label{thm:mu_to_mustar}
Let $\mu^*$ be as defined in Theorem~\ref{thm:optimum} and  $\boldsymbol{\mu}^n(t)$ be as defined in Theorem \ref{thm:process_mconv}. Then,
\begin{equation}
    \max_{v \in [n]} \left\lvert \mu_v^n(t) - \mu^* \right\rvert
    \to 0 \text{ as } n \to \infty \text{ and } t \to \infty,
\end{equation}
where the limits are taken either jointly, or first as $n\to\infty$ and then as $t\to\infty$.
\end{theorem}
From a high level, Theorem~\ref{thm:mu_to_mustar} states that $\mu^n(t)$, which is an asymptotic approximation of the original (stochastic) rate vector $\mu^{n,m}(t)$, converges to the optimal values in time, when the number of servers $n$ goes to infinity. This indicates that the stochastic algorithm will approximately solve the optimization problem given in Equation \eqref{eq:costfunction} when $m$ and $n$ are large.
The proof of Theorem~\ref{thm:mu_to_mustar} is given in Section \ref{sec:proof_mu_to_mustar}.

\section{Numerical results} \label{sec:sim}

In order to showcase the performance of our algorithm relative to the optimal rates, in this section we present extensive numerical experiments. 
We choose $g(\mu) = 1 / \mu$ and $h(\mu) = 0.1 \mu^2$ for the cost function, in which case the asymptotically optimal rate is $\mu^*=\sqrt[3]{5\lambda}$.

\paragraph{Validation of large-system asymptotics.}
First, we fix $\lambda=0.8$ and $m=500$, and we vary the number of servers from $n=2$ to $n=100$. In Figure \ref{fig:servers}
we can see that the average service rate in steady state becomes closer to the asymptotically optimal one as $n$ increases, being already quite close for $n=25$. The reason is that the bias in the estimator of the idle time vanishes as $n\to\infty$, and so the servers can accurate learn the optimal service rate. 

\begin{figure}[h!]
    \centering
    \includegraphics[width=0.49\textwidth]{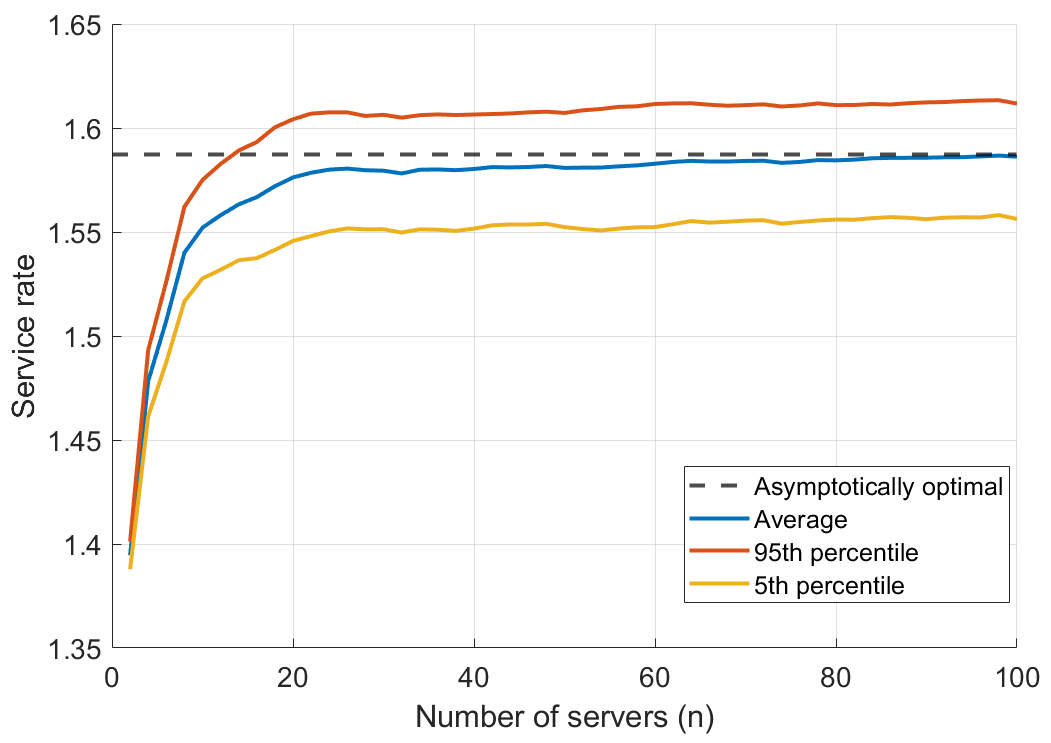}
    \caption{Steady state service rates for different number of servers $n$.}
    \label{fig:servers}
\end{figure}

\paragraph{Implications of a constant step-size.}
On the other hand, note that in Figure \ref{fig:servers} the 5th and 95th percentiles of the rates are bounded away from their average, even as $n$ increases. While this phenomenon is not apparent in our asymptotic analysis, where we started by taking the limit as the step size ($1/m$) goes to zero, it is the expected behavior of stochastic approximation algorithms with finite step size. 
Indeed, with such algorithms, the state converges to a steady state distribution around the optimal point, where the variance is an increasing function of the step size. To showcase this, we fix $\lambda=0.8$, $n=100$, and vary the the step size from $m=10$ to $m=500$. In Figure \ref{fig:step_size} we can see that the 5th and 95th percentiles become closer to the average, as $m$ increases. Since we chose $n=100$, the average rates almost coincide with the asymptotically optimal one.

\begin{figure}[h!]
     \centering
     \begin{subfigure}[b]{0.49\textwidth}
         \centering
         \includegraphics[width=\textwidth]{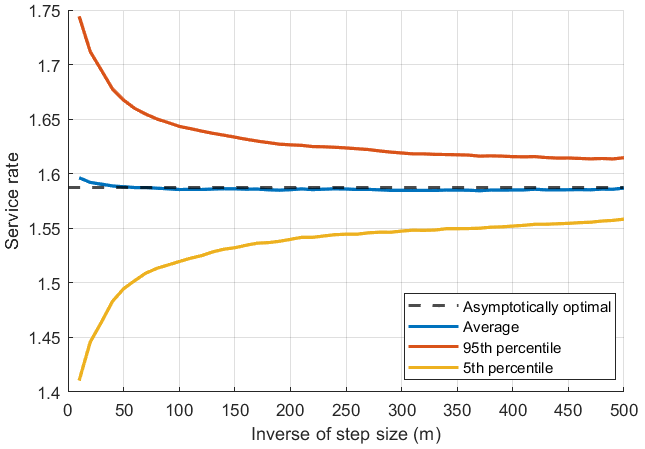}
         \caption{Steady state service rates for different values of the inverse step size ($m$).}
         \label{fig:step_size}
     \end{subfigure}
     \hfill
     \begin{subfigure}[b]{0.49\textwidth}
         \centering
         \includegraphics[width=\textwidth]{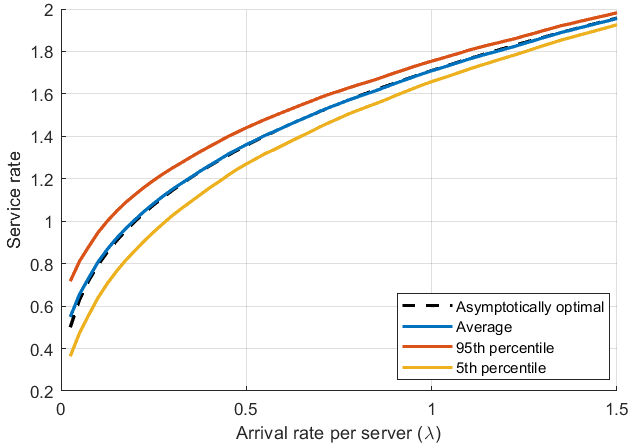}
         \caption{Steady state service rates for different values of the arrival rate per server ($\lambda$).}
         \label{fig:lambda}
     \end{subfigure}
        \caption{Influence of step size and arrival rate.}
        \label{fig:group-1}
\end{figure}


\paragraph{Influence of the arrival rate.}
Next, we fix $n=100$, $m=100$, and vary the arrival rate from $\lambda=0.025$ to $\lambda=1.5$. In Figure \ref{fig:lambda} we see that, as $\lambda$ increases, so does the average, in lockstep with the asymptotically optimal service rate, which is $\mu^*=\sqrt[3]{5\lambda}$. Moreover, we can see that the 5th and 95th percentiles become closer to the average as $\lambda$ increases. This is because, for higher values of $\lambda$, the derivatives of the objective function around the equilibrium are steeper, and thus the deviations around the average are reduced (interestingly, this is also observed in Stochastic Gradient Descent (SGD) with constant step sizes).



\paragraph{Extension to the infinite-buffer scenario.}
Finally, we verify that our algorithm also works as intended when we have infinite instead of unit buffers. In these simulations we use $n=50$, $m=500$, $\lambda=0.8$, and i.i.d. initial service rates $\mu_v(0) \sim \text{Unif}[0,2]$. In Figures \ref{fig:rate_unit_buffer} and \ref{fig:rate_inf_buffer} we see that, when the system starts empty, the trajectories of the rate vectors are indistinguishable between the cases of unit buffers and infinite buffers. This is because in the infinite buffer case, after a brief upward excursion in the queue lengths at the beginning while the system is overloaded (cf. Figure \ref{fig:queue_inf_buffer}), the maximum queue length is at most one, and so the extra buffer is not utilized. On the other hand, if the system starts with a sufficiently large backlog, it behaves as if it were overloaded until the backlog is handled. In that case, in Figure \ref{fig:inf_buffer_big}, we can see how the rates overshoot their optimal point while the queues remain large, but converge to the optimal rates once the backlog is resolved.
Establishing the convergence of the proposed algorithm formally in the infinite buffer case is an interesting future research direction.

\begin{figure}[h!]
     \centering
     \begin{subfigure}[b]{0.49\textwidth}
         \centering
         \includegraphics[width=\textwidth]{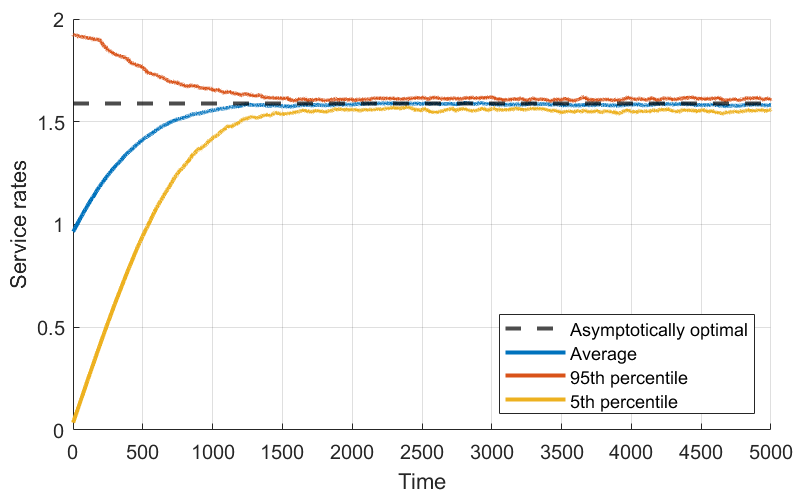}
         \caption{Service rates for the case of unit buffers, starting empty.}
         \label{fig:rate_unit_buffer}
     \end{subfigure}
     \hfill
     \begin{subfigure}[b]{0.49\textwidth}
         \centering
         \includegraphics[width=\textwidth]{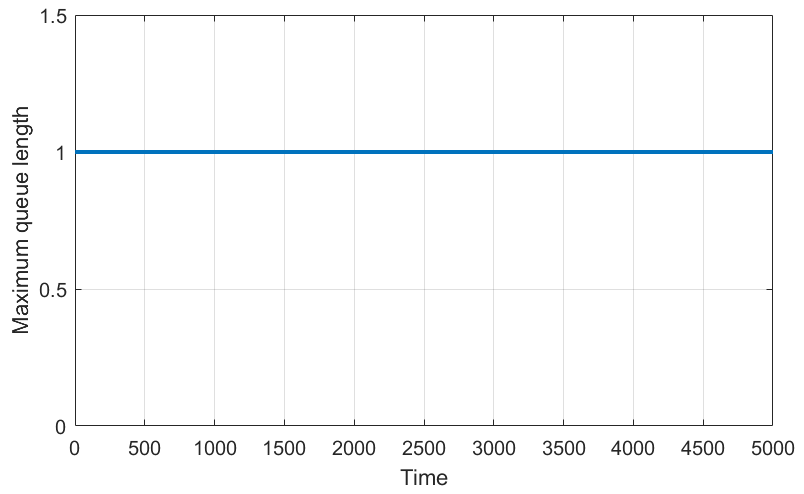}
         \caption{Maximum queue length for the case of unit buffers, starting empty.}
         \label{fig:queue_unit_buffer}
     \end{subfigure}
        \caption{Service rates and maximum queue lengths for the unit buffer case, starting empty.}
        \label{fig:unit_buffer}
\end{figure}

\begin{figure}[h!]
     \centering
     \begin{subfigure}[b]{0.49\textwidth}
         \centering
         \includegraphics[width=\textwidth]{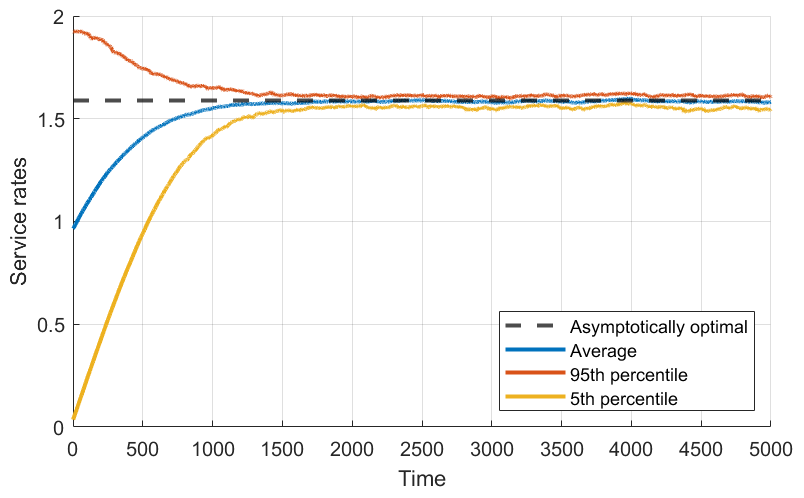}
         \caption{Service rates for the case of infinite buffers, starting empty.}
         \label{fig:rate_inf_buffer}
     \end{subfigure}
     \hfill
     \begin{subfigure}[b]{0.49\textwidth}
         \centering
         \includegraphics[width=\textwidth]{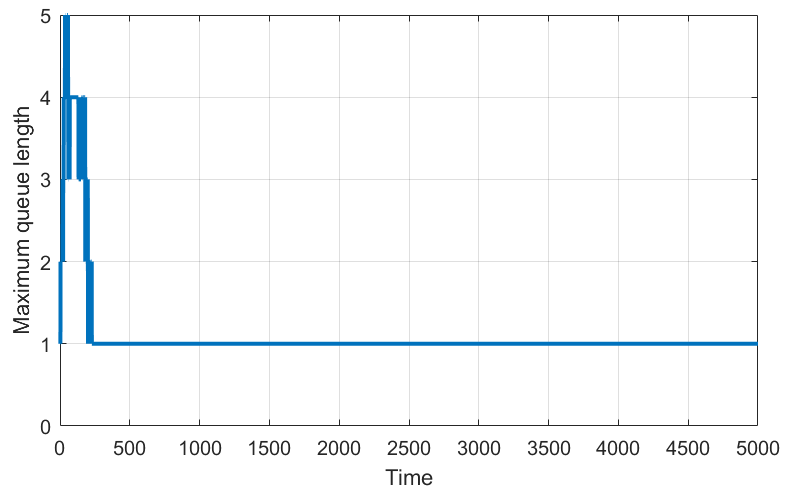}
         \caption{Maximum queue length for the case of infinite buffers, starting empty.}
         \label{fig:queue_inf_buffer}
     \end{subfigure}
        \caption{The service rates and maximum queue lengths for the infinite buffer case, starting empty.}
        \label{fig:inf_buffer}
\end{figure}

\begin{figure}[h!]
     \centering
     \begin{subfigure}[b]{0.49\textwidth}
         \centering
         \includegraphics[width=\textwidth]{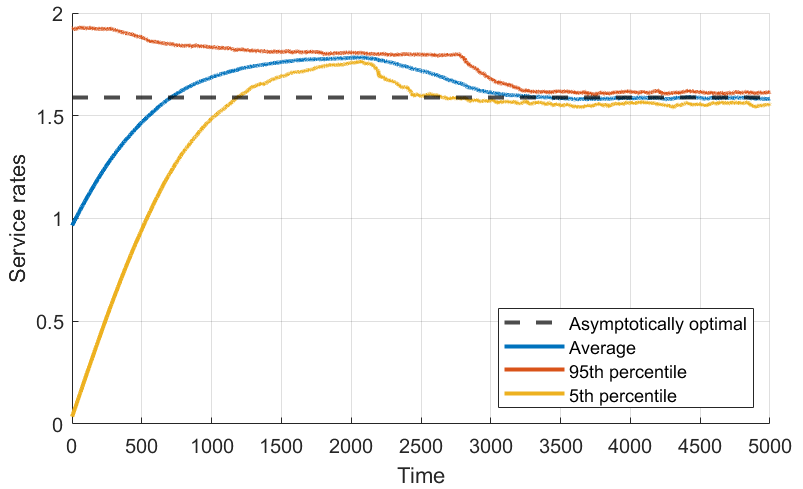}
         \caption{Service rates for the case of infinite buffers, starting with a backlog.}
         \label{fig:rate_inf_buffer_big}
     \end{subfigure}
     \hfill
     \begin{subfigure}[b]{0.49\textwidth}
         \centering
         \includegraphics[width=\textwidth]{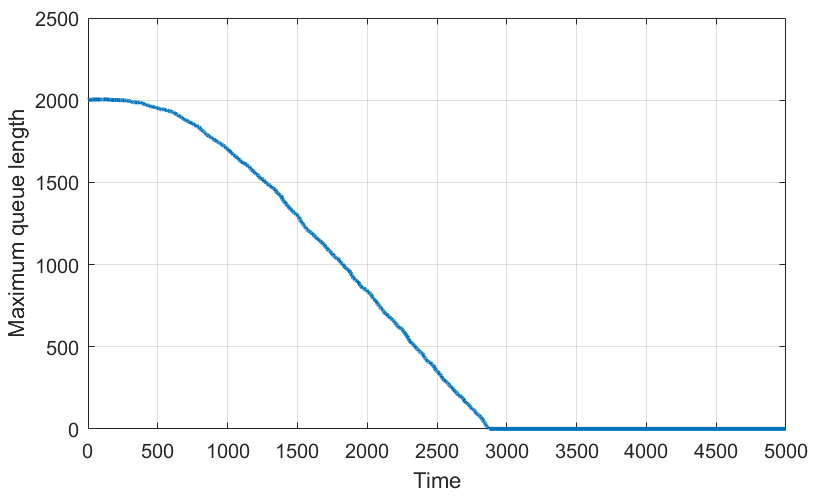}
         \caption{Maximum queue length for the case of infinite buffers, starting with a backlog.}
         \label{fig:queue_inf_buffer_big}
     \end{subfigure}
        \caption{Service rates and maximum queue lengths for the infinite buffer case, starting with a backlog.}
        \label{fig:inf_buffer_big}
\end{figure}

\section{Proofs of main results}\label{sec:proofs}

\subsection{Proof of Theorem \ref{thm:process_mconv}} \label{sec:proof_process_mconv}
The proof introduces a measure-valued representation of the process $\Big(\big(X^{n,m}_v(\cdot),I^{n,m}_v(\cdot)\big):v\in [n]\Big)$, where $X^{n,m}_v(t)\in\{0,1\}$ denotes whether server $v$ is idle or busy at time $t$, and $I^{n,m}_v(t)$ is the current idle time of server $v$ at time $t$. Using this measure-valued representation, the proof then relies on results in \cite{kurtz1992averaging} to establish the tightness of the sequence of these measures, and to argue that the limit points are integrals with respect to certain random measures. Then, in order to characterize these limiting measures we use a martingale decomposition of the pre-limit dynamics on certain test functions, and use Doob's maximal inequality to establish their limits.\\

Let $\nu^{n,m}$ be a random measure on $[0, \infty) \times \{ 0, 1 \}^n \times [0, \infty)^n$ defined as
\begin{equation}
    \nu^{n,m}\left( \Gamma \times \diff \boldsymbol{x} \times \diff \boldsymbol{i} \right)
    := \int\limits_\Gamma \delta_{\boldsymbol{X}^{n,m}(m s)}(\diff \boldsymbol{x}) \delta_{\boldsymbol{I}^{n,m}(m s)}(\diff \boldsymbol{i}) \diff s.
\end{equation}
Note that $\nu^{n,m}$ is the empirical measure of the queue length process $\boldsymbol{X}^{n,m}$ and the idle times $\boldsymbol{I}^{n,m}$.

\begin{lemma}
\label{lem:rel_compact}
The sequence $( \nu^{n,m} )_{m \in \N}$ is relatively compact.
\end{lemma}

\begin{proof}
Fix any $K > 0$. If $I_v^{n,m}(m t) > K$, then server $v \in [n]$ has not received any tasks throughout $[\tau - K, \tau]$, where $\tau \leq m t$ is the last time the server was idle or $\tau = m t$ if the server is currently idle. However, note that any server receives jobs at least at exponential rate $\lambda$. Thus, the probability that $I_v^{n,m}(m t) > K$ is upper bounded by $e^{-\lambda K}$. Therefore,
\begin{equation}
\begin{multlined}
    \E\left[ \nu^{n,m}([0, t] \times \{ 0, 1 \}^n \times [0, K]^n) \right] 
    = \E\left[ \int\limits_0^t \mathbbm{1}_{\big\{ I_v^{n,m}(m s) \leq K \text{ for all } v \in [n] \big\}} \diff s \right] \\
    = \int\limits_0^t \P\big( I_v^{n,m}(m s) \leq K \text{ for all } v \in [n] \big) \diff s
    \geq \int\limits_0^t \left( 1 - \sum_{v \in [n]} \P\left( I_v^{n,m}(m s) > K \right) \right) \diff s \\
    \geq \int\limits_0^t \left( 1 - n e^{-\lambda K} \right) \diff s
    = \left( 1 - n e^{-\lambda K} \right) t,
\end{multlined}
\end{equation}
for all $m \in \N$. It follows that the sequence $( \nu^{n,m} )_{m \in \N}$ is relatively compact by~\cite[Lemma 1.3]{kurtz1992averaging}.
\end{proof}

Since $( \nu^{n,m} )_{m \in \N}$ is relatively compact by Lemma \ref{lem:rel_compact}, then every subsequence has a convergent subsequence. Let $( \nu^{n,m} )_{m \in \N}$ be any such convergent subsequence, where we omit the subscripts of the subsequence for brevity, and let $\nu^{n,m} \to \nu^n$ as $m \to \infty$. Note that
\begin{equation}
\begin{multlined}
    \mu_v^{n,m}(m t) - \mu_v^{n,m}(0) 
    = - \frac{1}{m} \int\limits_0^{m t} \left[ g'(\mu_v^{n,m}(s)) + h'(\mu_v^{n,m}(s)) \left( I_v^{n,m}(s) + \frac{1}{\mu_v^{n,m}(s)} \right) \right] \diff s \\
    = - \int\limits_{[0, t] \times \{0, 1\}^n \times [0, \infty)^n} \left[ g'(\mu_v^{n,m}(m s)) + h'(\mu_v^{n,m}(m s)) \left( i_v + \frac{1}{\mu_v^{n,m}(s)} \right) \right] \nu^{n,m}(\diff s \times \diff \boldsymbol{x} \times \diff \boldsymbol{i}),
\end{multlined}
\end{equation}
and hence, along the convergent subsequence, $\boldsymbol{\mu}^{n,m}(m t) \to \boldsymbol{\mu}^{n}(t)$ weakly as $m \to \infty$ by the continuous mapping theorem, where
\begin{equation}
    \mu_v^{n}(t)
    = \mu_v^{n}(0) - \int\limits_{[0, t] \times \{0, 1\}^n \times [0, \infty)^n} \left[ g'(\mu_v^n(s)) + h'(\mu_v^n(s)) \left( i_v + \frac{1}{\mu_v^n(s)} \right) \right] \nu^{n}(\diff s \times \diff \boldsymbol{x} \times \diff \boldsymbol{i}).
\end{equation}
Moreover,  \cite[Lemma 1.4]{kurtz1992averaging} implies that there exists a family of random probability measures $\{\pi_t:t\geq 0\}$ on $\{ 0, 1 \}^n \times [0, \infty)^n$ such that
\begin{equation}
    \nu^{n}([0, t] \times \diff \boldsymbol{x} \times \diff \boldsymbol{i})
    = \int\limits_0^t \pi_s(\diff \boldsymbol{x} \times \diff \boldsymbol{i}) \diff s.
\end{equation}
To complete the proof of Theorem \ref{thm:process_mconv}, we should therefore show that $\pi_t$ is the stationary distribution of a system where the service rate is fixed at $\bmu^n(t)$. Let $f: \{ 0, 1 \}^n \times [0, \infty)^n \to \R$ be bounded and differentiable in its second coordinate. Then, using the convention $0/0=0$, we have 
\begin{equation}
\begin{multlined}
    f(\boldsymbol{X}^{n,m}(m t), \boldsymbol{I}^{n,m}(m t))
    = f(\boldsymbol{X}^{n,m}(0), \boldsymbol{I}^{n,m}(0)) \\
    + \sum_{v \in [n]} \int\limits_0^{m t} \frac{\partial}{\partial I^{n,m}_v} f(\boldsymbol{X}^{n,m}(s), \boldsymbol{I}^{n,m}(s)) (1 - X^{n,m}_v(s)) \diff s \\
    + \sum_{v \in [n]} \int\limits_0^{m t} ( f(\boldsymbol{X}^{n,m}(s) + e_v, \boldsymbol{I}^{n,m}(s)) \\ - f(\boldsymbol{X}^{n,m}(s), \boldsymbol{I}^{n,m}(s)) ) \diff N_v^{(1)}\left( \int\limits_0^s \frac{\lambda n (1 - X^{n,m}_v(u))}{\sum_{v' \in [n]} (1 - X^{n,m}_{v'}(u))} \diff u \right) \\
    + \sum_{v \in [n]} \int\limits_0^{m t} ( f(\boldsymbol{X}^{n,m}(s) - e_v, \boldsymbol{I}^{n,m}(s) - I^{n,m}_v(s) e_v) \\ - f(\boldsymbol{X}^{n,m}(s), \boldsymbol{I}^{n,m}(s) ) \diff N_v^{(2)}\left( \int\limits_0^s \mu_v^{n,m}(u) X^{n,m}_v(u) \diff u \right),
\end{multlined}
\end{equation}
where $N_v^{(1)}$ and $N_v^{(2)}$ are independent unit-rate Poisson processes and $e_v$ is the $n$-dimensional unit vector with a one at index $v$ and zeroes everywhere else, for $v \in [n]$. We divide by $m$ and rewrite the third and fourth terms on the right-hand to obtain
\begin{equation}
\begin{gathered}
\label{eq:process_mconv_1}
    \frac{f(\boldsymbol{X}^{n,m}(m t), \boldsymbol{I}^{n,m}(m t))}{m}
    = \frac{f(\boldsymbol{X}^{n,m}(0), \boldsymbol{I}^{n,m}(0))}{m} \\
    + \sum_{v \in [n]} \int\limits_0^t \frac{\partial}{\partial I^{n,m}_v} f(\boldsymbol{X}^{n,m}(m s), \boldsymbol{I}^{n,m}(m s)) (1 - X^{n,m}_v(m s)) \diff s \\
    + \sum_{v \in [n]} \int\limits_0^t ( f(\boldsymbol{X}^{n,m}(m s) + e_v, \boldsymbol{I}^{n,m}(m s)) \\ - f(\boldsymbol{X}^{n,m}(m s), \boldsymbol{I}^{n,m}(m s)) ) \frac{\lambda n (1 - X^{n,m}_v(m s))}{\sum_{v' \in [n]} (1 - X^{n,m}_{v'}(m s))} \diff s \\
    + \sum_{v \in [n]} \int\limits_0^t ( f(\boldsymbol{X}^{n,m}(m s) - e_v, \boldsymbol{I}^{n,m}(m s) - I^{n,m}_v(m s) e_v) \\ - f(\boldsymbol{X}^{n,m}(m s), \boldsymbol{I}^{n,m}(m s)) ) \mu_v^{n,m}(m s) X^{n,m}_v(m s) \diff s \\
    + \frac{M^{n,m}_{(1)}(m t)}{m} + \frac{M^{n,m}_{(2)}(m t)}{m},
\end{gathered}
\end{equation}
where
\begin{equation}
\begin{aligned}
    M^{n,m}_{(1)}(t) &:= \sum_{v \in [n]} \int\limits_0^t ( f(\boldsymbol{X}^{n,m}(s) + e_v, \boldsymbol{I}^{n,m}(s)) \\ &\hspace{3cm} - f(\boldsymbol{X}^{n,m}(s), \boldsymbol{I}^{n,m}(s)) ) \diff M_v^{(1)}\left( \int\limits_0^s \frac{\lambda n (1 - X^{n,m}_v(u))}{\sum_{v' \in [n]} (1 - X^{n,m}_{v'}(u))} \diff u \right), \\
    M^{n,m}_{(2)}(t) &:= \sum_{v \in [n]} \int\limits_0^t ( f(\boldsymbol{X}^{n,m}(s) - e_v, \boldsymbol{I}^{n,m}(s) - I^{n,m}_v(s) e_v) \\ &\hspace{3cm} - f(\boldsymbol{X}^{n,m}(s), \boldsymbol{I}^{n,m}(s)) ) \diff M_v^{(2)}\left( \int\limits_0^s \mu_v^{n,m}(s) X^{n,m}_v(s) \diff u \right),
\end{aligned}
\end{equation}
and $M_v^{(i)}(t) := N_v^{(i)}(t) - t$ for $i = 1, 2$ and $v \in [n]$. Let $\mathcal{F}_t$ be the natural filtration of $(\boldsymbol{X}^{n,m}(t), \boldsymbol{I}^{n,m}(t))$. Then, it is easily checked that $M^{n,m}_{(1)}(t)$ and $M^{n,m}_{(2)}(t)$ are square-integrable martingales with respect to $\mathcal{F}_t$. Moreover,
\begin{equation}
\begin{aligned}
    \E\left[ \left\langle M^{n,m}_{(1)}, M^{n,m}_{(1)} \right\rangle(m t) \right] 
    &\leq 2 \lVert f \rVert_\infty \sum_{v \in [n]} \E\left[ \left\langle N_v^{(1)}, N_v^{(1)} \right\rangle \left( \int\limits_0^{m t} \frac{\lambda n (1 - X^{n,m}_v(u))}{\sum_{v' \in [n]} (1 - X^{n,m}_{v'}(u))} \diff s \right) \right] 
   \\ &= 2 \lVert f \rVert_\infty \sum_{v \in [n]} \E\left[ \int\limits_0^{m t} \frac{\lambda n (1 - X^{n,m}_v(u))}{\sum_{v' \in [n]} (1 - X^{n,m}_{v'}(u))} \diff s \right]
    = 2 \lVert f \rVert_\infty \lambda n m t,
\end{aligned}
\end{equation}
and
\begin{equation}
\begin{aligned}
    \E\left[ \left\langle M^{n,m}_{(2)}, M^{n,m}_{(2)} \right\rangle(m t) \right] 
    &\leq 2 \lVert f \rVert_\infty \sum_{v \in [n]} \E\left[ \left\langle N_v^{(2)}, N_v^{(2)} \right\rangle \left( \int\limits_0^{m t} \mu_v(s) X^{n,m}_v(s) \diff s \right) \right] \\
    &= 2 \lVert f \rVert_\infty \sum_{v \in [n]} \E\left[ \int\limits_0^{m t} \mu_v^{n,m}(s) X^{n,m}_v(s) \diff s \right]
    \leq 2 \lVert f \rVert_\infty \mu_+ n m t.
\end{aligned}
\end{equation}
Therefore, by Doob's maximal inequality, $M^{n,m}_{(1)}(t) / m \to 0$ and $M^{n,m}_{(2)}(t) / m \to 0$, as $m \to \infty$ uniformly on any bounded time interval. Also, $f(\boldsymbol{X}^{n,m}(m t), \boldsymbol{I}^{n,m}(m t)) / m \to 0$ and $f(\boldsymbol{X}^{n,m}(0), \boldsymbol{I}^{n,m}(0)) / m \to 0$, as $m \to \infty$ uniformly on any bounded time interval. Combining these facts with Equation \eqref{eq:process_mconv_1}, we get
\begin{equation}
\begin{multlined}
    \int\limits_{[0, t] \times \{ 0, 1 \}^n \times [0, \infty)^n} \sum_{v \in [n]} \Bigg( \frac{\partial}{\partial i_v} f(\boldsymbol{x}, \boldsymbol{i}) (1 - x_v)
    + \left( f(\boldsymbol{x} + e_v, \boldsymbol{i}) - f(\boldsymbol{x}, \boldsymbol{i}) \right) \frac{\lambda n (1 - x_v)}{\sum_{v' \in [n]} (1 - x_{v'})} \\
    + \left( f(\boldsymbol{x} - e_v, \boldsymbol{i} - i_v e_v) - f(\boldsymbol{x}, \boldsymbol{i}) \right) \mu_v^{n,m}(m s) x_v \Bigg) \nu^{n,m}(\diff s \times \diff \boldsymbol{x} \times \diff \boldsymbol{i})
    \to 0,
\end{multlined}
\end{equation}
as $m \to \infty$ uniformly on any bounded time interval, and in particular along the convergent subsequence. As a result, by the continuous mapping theorem, we have
\begin{equation}
\begin{aligned}
    \int\limits_{\{ 0, 1 \}^n \times [0, \infty)^n} \sum_{v \in [n]} \Bigg( \frac{\partial}{\partial i_v} f(\boldsymbol{x}, \boldsymbol{i}) (1 - x_v)
    &+ \left( f(\boldsymbol{x} + e_v, \boldsymbol{i}) - f(\boldsymbol{x}, \boldsymbol{i}) \right) \frac{\lambda n (1 - x_v)}{\sum_{v' \in [n]} (1 - x_{v'})} \\
    &+ \left( f(\boldsymbol{x} - e_v, \boldsymbol{i} - i_v e_v) - f(\boldsymbol{x}, \boldsymbol{i}) \right) \mu_v^n(t) x_v \Bigg) \pi_t(\diff \boldsymbol{x} \times \diff \boldsymbol{i})
    = 0.
\end{aligned}
\end{equation}
Thus, \cite[Proposition 9.2]{kurtz1986markov} implies that $\pi_t$ is the stationary distribution of a system where the service rate is fixed at $\boldsymbol{\mu}^n(t)$. Note that the stationary distribution is unique since the process is irreducible. This concludes the proof of Theorem \ref{thm:process_mconv}.

\subsection{Proof of Theorem \ref{thm:idle}} \label{sec:proof_subcrit_idle}

As $n$ is fixed throughout, we will omit the dependence on $n$ in the notation. A server $v \in [n]$ thus processes tasks at a fixed rate $\mu_v$. Let 
\begin{equation}\label{eq:delta}
    \delta:= \frac{1}{n} \sum_{v \in [n]} \mu_v - \lambda
\end{equation}
be the amount of excess processing power (per-server) in the system. We prove the theorem for two cases separately: for the case where the system is in a supercritical regime (i.e., when $\delta \leq 0$), and for the case where it is in a subcritical regime (i.e., when $\delta > 0$). In both cases, we first show that the fraction of busy servers concentrate around a constant (as $n\to\infty$), and then argue that this implies that the expected idle times become all equal (as $n\to\infty$).

\subsubsection{Supercritical regime}
Throughout this subsection we consider the case where $\delta\leq 0$ in Equation~\eqref{eq:delta},
that is, where the system is either critically loaded, or overloaded. In this case, we first show that the fraction of busy servers concentrate around $1$ (i.e., around the state where all servers are busy), as $n\to\infty$.
Recall $\mu_-$ from Assumption~\ref{ass:opt}.
\begin{lemma}
\label{thm:supercrit_bound}
Let $\varepsilon$ be such that $0< \varepsilon \leq 1 / 3$. Then,
\begin{equation}
    \P\left( \frac{1}{n} \sum_{v \in [n]} X_v(\infty) \leq 1 - 3 \varepsilon \right)
    \leq \left( 1 - \frac{\mu_- \varepsilon}{\lambda} \right)^{\varepsilon n}.
\end{equation}
\end{lemma}
This results follows by lower bounding the fraction of busy servers by an appropriate birth-death process. The proof is given in Appendix \ref{sec:supercrit_bound}.\\

Leveraging this concentration result, we obtain an upper bound on the expected idle times.

\begin{lemma}
\label{thm:supercrit_idle}
Let $\varepsilon$ be such that $0< \varepsilon \leq 1 / 3$. Then,
\begin{equation}\label{eq:aux}
    \E\left[ I_v(\infty) \right]
    \leq \frac{6 \varepsilon}{\lambda}
    + \left( \frac{\sqrt{2}}{\lambda} + \frac{1}{\mu_-} \right) \left( 1 - \frac{\mu_- \varepsilon}{\lambda} \right)^{\varepsilon n / 2}.
\end{equation}
\end{lemma}
We establish Lemma~\ref{thm:supercrit_idle} by expressing the expected idle times as a function of the fraction of busy servers, and using the concentration result of Lemma \ref{thm:supercrit_bound} to bound it. The proof is given in Appendix \ref{sec:supercrit_idle}.\\

Taking the limit as $n\to\infty$ in Equation \eqref{eq:aux} yields
    $\lim\limits_{n\to\infty} \E\left[ I_v(\infty) \right] \leq 6 \varepsilon/\lambda.$
Since this holds for all $\varepsilon>0$ small enough, we have
$
    \lim\limits_{n\to\infty} \E\left[ I_v(\infty) \right] =0,
$
which concludes the proof of Theorem \ref{thm:idle} for the supercritical case.

\subsubsection{Subcritical regime}
Throughout this subsection we consider the case where $\delta> 0$ in Equation~\eqref{eq:delta},
that is, where the system is in a subcritical regime. In this case, we also need to show that the fraction of busy servers concentrate around a constant, as $n\to\infty$. However, here the constant is bounded away from zero, and thus the concentration result is significantly more involved.
In order to prove our desired concentration result, we define an auxiliary measure-valued process. Let
\[ \Phi(\diff x) := \frac{1}{n} \sum_{v \in [n]} \delta_{\mu_v}(\diff x) \]
and let $\bar{\phi}_t$ be a measure-valued process on $\R_+$ such that
\begin{equation}\label{eq:phi-t}
    \int\limits_0^\infty f(x) \bar{\phi}_t(\diff x)
    = \int\limits_0^\infty f(x) \bar{\phi}_0(\diff x) 
    + \int\limits_0^t \left( \frac{\lambda \left( \int_0^\infty f(x) \Phi(\diff x) - \int_0^\infty f(x) \bar{\phi}_s(\diff x) \right)}{1 - \int_0^\infty \bar{\phi}_s(\diff x)} - \int\limits_0^\infty x f(x) \bar{\phi}_s(\diff x) \right) \diff s,
\end{equation}
for all $f: \R_+ \to [0, 1]$, with
\[ \bar{\phi}_0(\diff x) = \frac{1}{n} \sum_{v \in [n]} \delta_{\mu_v} (\diff x) X_v(0). \]

\begin{remark}\normalfont
It can be easily checked that, for all $t\geq 0$, the measure $\bar{\phi}_t$ is a collection of $n$ point masses at $\mu_v$, for $v\in[n]$. Hence, its existence and uniqueness follow by standard arguments using sample path constructions.
\end{remark}

We first show that, when $n$ is large, the trajectory of the state of the system is ``close" (in a strong sense) to the auxiliary measure-valued process.

\begin{lemma}
\label{thm:subcrit_transient_limit}
Suppose that
$\int_0^\infty \bar{\phi}_0(\diff x) \leq 1 - \frac{\delta}{\mu_+}.$
Then, for all $T \geq 0$, we have
\begin{equation}\label{eq:strong_bound}
\begin{multlined}
    \sup_{f: \R_+ \to [0, 1]} \E\left[ \sup_{t \in [0, T]} \left\lvert \frac{1}{n} \sum_{v \in [n]} f(\mu_v) X_v(t) - \int_0^\infty f(x) \bar{\phi}_t(\diff x) \right\rvert \right]  \\
    \leq \sqrt{\frac{8 (\mu_+ + \lambda) T}{n}} \exp\left( \left( \frac{6 \lambda \mu_+}{\delta} + \mu_+ \right) T \right).
\end{multlined}
\end{equation}
\end{lemma}
The proof of Lemma~\ref{thm:subcrit_transient_limit} consists of considering a martingale decomposition of the state process, 
showing that the drift is close to our auxiliary process, and using Doob's maximal inequality to bound the corresponding martingale. The proof is given in Appendix \ref{sec:subcrit_transient_limit}.\\

The next step in the proof of Theorem \ref{thm:idle} for the subcritical case is to use the auxiliary process to obtain a concentration abound on the fraction of idle servers. In order to do this, we first establish the following monotonicity result.

\begin{lemma}
\label{lem:stoch_dom}
Let $\boldsymbol{X}^{(1)}(t)$ and $\boldsymbol{X}^{(2)}(t)$ be two copies of the queue length process such that stochastically $\boldsymbol{X}^{(1)}(0) \leq \boldsymbol{X}^{(2)}(0)$, where inequality is considered coordinatewise. Then, there exists a joint probability space such that $\boldsymbol{X}^{(1)}(t) \leq \boldsymbol{X}^{(2)}(t)$ for all $t \geq 0$, almost surely.
\end{lemma}
Lemma~\ref{lem:stoch_dom} follows from constructing an appropriate coupling between the processes, and showing that the order is maintained across time. The proof is given in Appendix \ref{sec:stoch_dom}.\\

Using the monotonicity given in Lemma \ref{lem:stoch_dom}, we obtain the following exponential mixing time result.

\begin{lemma} \label{thm:subcrit_mixing_time}
Let $\boldsymbol{X}^{(1)}(t)$ and $\boldsymbol{X}^{(2)}(t)$ be two copies of the queue length process such that stochastically $\boldsymbol{X}^{(1)}(0) \leq \boldsymbol{X}^{(2)}(0)$, where inequality is considered coordinatewise. Then, there exists a joint probability space such that
\begin{equation}
    \E\left[ \frac{1}{n} \sum_{v \in [n]} \left\lvert X_v^{(2)}(t) - X_v^{(1)}(t) \right\rvert \right] \leq \exp\left( -\mu_- t \right).
\end{equation}
\end{lemma}
Lemma~\ref{thm:subcrit_mixing_time} follows from the monotonicity of Lemma \ref{lem:stoch_dom}, and applying Gr\"{o}nwall's inequality on the dynamics of the system. The proof is given in Appendix \ref{sec:subcrit_mixing_time}.\\

We now state the concentration result on the fraction of busy servers, as mentioned before.

\begin{lemma}
\label{thm:subcrit_concentrate}
Let $\varepsilon$ be such that $0 < \varepsilon < \delta / \mu_-$. Then,
\begin{equation}
    \E\left[ \left\lvert \frac{1}{n} \sum_{v \in [n]} X_v(\infty) - c(\bmu) \right\rvert \right]
    \leq \left(1 + \sqrt{\frac{8 (\mu_+ + \lambda) \log(n)}{\alpha}}\right)n^{-\frac{\mu_-}{\alpha}},
\end{equation}
where
\[ \alpha := 2 \left(\mu_- + \frac{6 \lambda \mu_+}{\varepsilon \mu_-} + \mu_+\right) \qquad \text{ and } \qquad c(\bmu) := \int\limits_0^\infty \bar{\phi}_{\frac{\log(n)}{\alpha}}(\diff x), \]
for $\bar{\phi}_t$ as defined in Equation \eqref{eq:phi-t}, with $\bar{\phi}_0(\diff x) = 0$.
\end{lemma}
Lemma~\ref{thm:subcrit_concentrate} is established by considering two coupled systems, one starting empty and one starting in steady state. On the one hand, Lemma \ref{thm:subcrit_mixing_time} implies that these two process converge to each other exponentially fast in time. On the other hand, Lemma \ref{thm:subcrit_transient_limit} implies that the system that starts empty is close to the integral of the measure-valued auxiliary process in Equation~\eqref{eq:phi-t} (which is deterministic), for $n$ sufficiently large. Therefore, combining these two results, we bound the distance between the fraction of servers in steady state, and the integral of the auxiliary process at an appropriately chosen time (at which point it is equal to $c(\bmu)$). The proof is given in Appendix \ref{sec:subcrit_concentrate}.\\

Before stating and proving the main theorem for the subcritical case, we first need to establish that the probability of all servers being busy is exponentially small in $n$.

\begin{lemma}
\label{thm:subcrit_bound}
Let $\varepsilon$ be such that $0<\varepsilon \leq \delta / (3 \mu_+)$. Then,
\begin{equation}
    \P\left( \frac{1}{n} \sum_{v \in [n]} X_v(\infty) \geq 1 - \varepsilon \right)
    \leq \left( 1 - \frac{\mu_+ \varepsilon}{\lambda + \mu_+ \varepsilon} \right)^{\varepsilon n}.
\end{equation}
\end{lemma}
Similarly to Lemma \ref{thm:supercrit_bound}, we establish Lemma~\ref{thm:subcrit_bound} by coupling the system with an appropriate birth-death process, and establishing the result for this simpler process. The proof is given in Appendix \ref{sec:subcrit_bound}.\\

Leveraging the concentration results in Lemmas~\ref{thm:supercrit_bound} and~\ref{thm:subcrit_bound}, in Lemma~\ref{thm:subcrit_idle} below we obtain a bound on how much the expected delays deviate from appropriate constants.

\begin{lemma}
\label{thm:subcrit_idle}
Let $\varepsilon$ be such that $0< \varepsilon \leq \delta / \mu_-$. Then,
\begin{equation}\label{eq:aux2}
\begin{multlined}
    \left\lvert \E\left[ I_v(\infty) \right] - \frac{1 - c(\bmu)}{\lambda} \right\rvert
    \leq \frac{2 \varepsilon}{\lambda}
    + \left( \frac{3 \sqrt{2} \mu_+}{\varepsilon \lambda \mu_-} + \frac{3 \mu_+}{\varepsilon \mu_-^2} \right) \sqrt{\frac{1 + \sqrt{\frac{8 (\mu_+ + \lambda) \log(n)}{\alpha}}}{n^{\frac{\mu_-}{\alpha}}}} \\
    + \left( \frac{\sqrt{2} n}{\lambda} + \frac{n}{\mu_-} \right) \left( 1 - \frac{\varepsilon \mu_-}{3 \lambda + \varepsilon \mu_-} \right)^{\varepsilon \mu_- n / (6 \mu_+)},
\end{multlined}
\end{equation}
where $c(\bmu)$ and $\alpha$ are as defined in Lemma \ref{thm:subcrit_concentrate}.
\end{lemma}
The proof of Lemma~\ref{thm:subcrit_idle} is given in Appendix \ref{sec:subcrit_idle}.\\

Now, taking the limit as $n\to\infty$ in Equation \eqref{eq:aux2} yields
\begin{equation}
    \lim\limits_{n\to\infty} \left\lvert \E\left[ I_v(\infty) \right] - \frac{1 - c(\bmu)}{\lambda} \right\rvert
    \leq \frac{2 \varepsilon}{\lambda}.
\end{equation}
Since this holds for all $\varepsilon>0$ small enough, we have
$$
    \lim\limits_{n\to\infty} 
    \E\left[ I_v(\infty) \right] = \frac{1 - c(\bmu)}{\lambda}.
$$
The only thing left in the proof of Theorem \ref{thm:idle} is to obtain the bounds on the constants $c(\bmu)$. For this purpose, we have the following result.

\begin{lemma}
\label{lem:barphi_bounded}
Let $\bar{\phi}_t$ be as defined in Lemma \ref{thm:subcrit_transient_limit}. Then,
\begin{equation}
\begin{aligned}
    \int_0^\infty \bar{\phi}_t(\diff x)
    &\geq\frac{\lambda}{\max_{v \in [n]} \mu_v} - \left\lvert \int_0^\infty \bar{\phi}_0(\diff x) - \frac{\lambda}{\max_{v \in [n]} \mu_v} \right\rvert \exp(-\max_{v \in [n]} \mu_v t) \\
    \int_0^\infty \bar{\phi}_t(\diff x)
    &\leq \frac{\lambda}{\min_{v \in [n]} \mu_v} + \left\lvert \int_0^\infty \bar{\phi}_0(\diff x) - \frac{\lambda}{\min_{v \in [n]} \mu_v} \right\rvert \exp(-\min_{v \in [n]} \mu_v t).
\end{aligned}
\end{equation}
\end{lemma}
The proof of Lemma~\ref{lem:barphi_bounded} consists of defining simple upper and lower bounds for the process of interest and bounding them, and it is given in Appendix \ref{sec:barphi_bounded}.\\

Finally, Lemma~\ref{thm:subcrit_concentrate} states that
\[ c(\bmu)=\int\limits_0^\infty \bar{\phi}_{\frac{\log(n)}{\alpha}}(dx), \]
with $\int_0^\infty \bar{\phi}_0(dx)=0$. Then, applying Lemma \ref{lem:barphi_bounded} for $t=\log(n)/\alpha$ and $\int_0^\infty \bar{\phi}_0(dx)=0$, and taking the limit as $n\to\infty$ yields bounds on $c(\bmu)$.
This completes the proof of Theorem \ref{thm:idle}.

\subsection{Proof of Lemma \ref{lem:maxmin_conv}} \label{sec:proof_maxmin_conv}

Theorem \ref{thm:idle} implies that there exist constants $0 \leq c^n(\bmu^n) \leq 1$ and a function $\varepsilon(n)\in o(1)$ such that
\begin{equation}
    \left\lvert \E_{\bmu^n}\left[ I_v^n(\infty) \right] - \frac{1 - c^n(\bmu^n)}{\lambda} \right\rvert \leq \varepsilon(n),
\end{equation}
for all $n \geq 1$.
Also, as discussed in Remark~\ref{rem:mu-up-low}, $\bmu^n \in [\mu_-, \mu_+]^n$. Then, for all $n \geq 1$ and $t \geq 0$,
\begin{equation}
\begin{multlined}
    \frac{\diff}{\diff t} \mu_v^n(t)
    = -g'(\mu_v^n(t)) - h'(\mu_v^n(t)) \left( \E_{\boldsymbol{\mu}^n(t)}[I_v^n(\infty)] + \frac{1}{\mu_v^n(t)} \right) \\
    = -g'(\mu_v^n(t)) - h'(\mu_v^n(t)) \left( \frac{1 - c^n(\boldsymbol{\mu}^n(t))}{\lambda} + \frac{1}{\mu_v^n(t)} \right) \pm \varepsilon(n) \lVert h' \rVert_\infty,
\end{multlined}
\end{equation}
where $x =y\pm z$ means that $x\in [y-z, y+z]$. Therefore, since Assumption \ref{ass:opt} implies that $g'$ and $h'$ are increasing, and $h'(x)/x$ is non-decreasing, we have
\begin{equation}
\begin{multlined}
    \frac{\diff}{\diff t} \left( \max_{v \in [n]} \mu_v^n(t) - \min_{v \in [n]} \mu_v^n(t) \right)
    \leq - \left( g'\left( \max_{v \in [n]} \mu_v^n(t) \right) - g'\left( \min_{v \in [n]} \mu_v^n(t) \right) \right) \\ 
    - \frac{1 - c^n(\boldsymbol{\mu}^n(t))}{\lambda} \left( h'\left( \max_{v \in [n]]} \mu_v^n(t) \right) - h'\left( \min_{v \in [n]} \mu_v^n(t) \right) \right) + 2 \varepsilon(n) \lVert h' \rVert_\infty \\
    \leq -\left( \sigma_g + \frac{\sigma_h \left( 1 - c^n(\boldsymbol{\mu}^n(t)) \right)}{\lambda} \right) \left( \max_{v \in [n]} \mu_v^n(t) - \min_{v \in [n]} \mu_v^n(t) \right) + 2 \varepsilon(n) \lVert h' \rVert_\infty.
\end{multlined}
\end{equation}
It follows that
\begin{equation}
\begin{multlined}
    \max_{v \in [n]} \mu_v^n(t) - \min_{v \in [n]} \mu_v^n(t) \leq \exp\left( -\left( \sigma_g + \frac{\sigma_h \left( 1 - c^n(\boldsymbol{\mu}^n(t)) \right)}{\lambda} \right) t \right) 
    + \frac{2 \varepsilon(n) \lVert h' \rVert_\infty}{\sigma_g + \frac{\sigma_h \left( 1 - c^n(\boldsymbol{\mu}^n(t)) \right)}{\lambda}}.
\end{multlined}
\end{equation}
Taking the limit as $t$ and $n$ goes to infinity concludes the proof of Lemma~\ref{lem:maxmin_conv}.

\subsection{Proof of Theorem \ref{thm:mu_to_mustar}} \label{sec:proof_mu_to_mustar}

Lemma \ref{lem:maxmin_conv} implies that there exists a function $\epsilon_d(n,t)$ that converges to zero as $n$ and $t$ go to infinity (regardless of the order in which these limits are taken), such that
\begin{equation}
    \max_{v \in [n]} \mu_v^n(t) - \min_{v \in [n]} \mu_v^n(t) \leq \varepsilon_d(n,t),
\end{equation}
for all $n \geq 1$ and $t \geq 0$. Now, we claim that there exist a constant $c_1 > 0$ and a function $\varepsilon(n,t)$ that converges to zero as $n$ and $t$ go to infinity (regardless of the order in which these limits are taken), such that
\begin{equation}\label{eq:gron-2.10}
    \frac{\diff}{\diff t} \left\lvert \lambda g'(\mu_v^n(t)) + h'(\mu_v^n(t)) \right\rvert
    \leq - c_1 \left\lvert \lambda g'(\mu_v^n(t)) + h'(\mu_v^n(t)) \right\rvert + \varepsilon(n,t),
\end{equation}
for all $n$ and $t$ sufficiently large, as long as $\mu_v^n(t)\neq \mu^*$ (which is the only value for which $\lambda g'(\mu_v^n(t)) + h'(\mu_v^n(t))=0$). This claim is sufficient to complete the proof of Theorem \ref{thm:mu_to_mustar}. To prove the claim, we distinguish again whether the system is in the supercritical or the subcritical regime. That is, we distinguish whether 
$\delta$ in Equation~\eqref{eq:delta}
is non-positive or positive, respectively.

\subsubsection{Supercritical regime.}
In this subsection, we prove Equation~\eqref{eq:gron-2.10} in the supercritical regime.
Note that Theorem \ref{thm:idle} implies that there exists a function $\varepsilon_I(n)$ with $\varepsilon_I(n)\to 0$ as $n\to\infty$ such that
\begin{equation}\label{eq:aux_I}
    \E_{\bmu^n}\left[ I_v^n(\infty) \right] \leq \varepsilon_I(n),
\end{equation}
for all $n \geq 1$. Moreover, since $\delta\leq 0$, we have
\begin{equation}\label{eq:aux_mu}
    \mu_v^n(t)
    \leq \min_{v \in [n]} \mu_v^n(t) + \varepsilon_d(n,t)
    \leq \lambda + \varepsilon_d(n,t)
    \leq \lambda + (1 + \mu_-) \varepsilon_d(n,t).
\end{equation}
Therefore, we have
\begin{equation}\label{eq:aux_dmu}
\begin{aligned}
    \frac{\diff}{\diff t} \mu_v^n(t) 
    &= - g'(\mu_v^n(t)) - h'(\mu_v^n(t)) \left( \E_{\boldsymbol{\mu}^n(t)}\left[ I_v^n(\infty) \right] + \frac{1}{\mu_v^n(t)} \right) 
    \geq - g'(\mu_v^n(t)) - \lVert h' \rVert_\infty \varepsilon_I(n) - \frac{h'(\mu_v^n(t))}{\mu_v^n(t)} \\
    &\geq -g'(\mu_v^n(t)) - \lVert h' \rVert_\infty \varepsilon_I(n) - \frac{h'(\lambda + (1 + \mu_-) \varepsilon_d(n,t))}{\lambda + (1 + \mu_-) \varepsilon_d(n,t)} \\
    &\geq -g'(\lambda) - \lVert g'' \rVert_\infty (1 + \mu_-) \varepsilon_d(n,t) - \lVert h' \rVert_\infty \varepsilon_I(n) - \frac{h'(\lambda)}{\lambda}  \\
    &\hspace{3cm}- \left( \frac{\lVert h'' \rVert_\infty}{\lambda} - \frac{h'(\lambda)}{\lambda (\lambda + (1 + \mu_-) \varepsilon_d(n,t))} \right) (1 + \mu_-) \varepsilon_d(n,t),
\end{aligned}
\end{equation}
where the first inequality follows from Equation \eqref{eq:aux_I}, the second one follows from Equation \eqref{eq:aux_mu} and the fact that $h'(x) / x$ is non-decreasing in $x$, and the third one follows from the convexity of $g$ and $h$. Also, combining the fact that
 $\lambda g'(\lambda) + h'(\lambda) < 0$
by Assumption \ref{ass:opt} with Equation \eqref{eq:aux_mu}, we get that
\begin{equation}
    \lambda g'(\mu_v^n(t)) + h'(\mu_v^n(t)) < 0,
\end{equation}
for all $n$ and $t$ sufficiently large. Combining this with Equation \eqref{eq:aux_dmu}, we obtain
\begin{equation}
\begin{multlined}
    \frac{\diff}{\diff t} \left\lvert \lambda g'(\mu_v^n(t)) + h'(\mu_v^n(t)) \right\rvert
    = - \left( \lambda g''(\mu_v^n(t)) + h''(\mu_v^n(t)) \right) \frac{\diff}{\diff t} \mu_v^n(t) \\
    \leq - \left( \sigma_g + \frac{\sigma_h}{\lambda} \right) \left\lvert \lambda g'(\lambda) + h'(\lambda) \right\rvert + \lVert g'' \rVert_\infty (1 + \mu_-) \varepsilon_d(n,t) + \lVert h' \rVert_\infty \varepsilon_I(n) \\ 
    \hspace{3cm}+ \left( \frac{\lVert h'' \rVert_\infty}{\lambda} - \frac{h'(\lambda)}{\lambda (\lambda + (1 + \mu_-) \varepsilon_d(n,t))} \right) (1 + \mu_-) \varepsilon_d(n,t) \\
    \leq - \left( \sigma_g + \frac{\sigma_h}{\lambda} \right) \left\lvert \lambda g'(\mu_v^n(t)) + h'(\mu_v^n(t)) \right\rvert + \lVert g'' \rVert_\infty (1 + \mu_-) \varepsilon_d(n,t) + \lVert h' \rVert_\infty \varepsilon_I(n) \\ 
    \hspace{3cm}+ \left( \frac{\lVert h'' \rVert_\infty}{\lambda} - \frac{h'(\lambda)}{\lambda (\lambda + (1 + \mu_-) \varepsilon_d(n,t))} \right) (1 + \mu_-) \varepsilon_d(n,t),
\end{multlined}
\end{equation}
for all $n$ and $t$ sufficiently large, where the second inequality follows since $g'(x)$ and $h'(x)$ are continuous and $\mu_v^n(t)$ is bounded, which completes the proof of the claim.

\subsubsection{Subcritical regime.}
In this subsection, we prove Equation~\eqref{eq:gron-2.10} in the subcritical regime.
Theorem \ref{thm:idle} implies that there exists constants $0 \leq c^n(\bmu^n) \leq 1$ and a function $\varepsilon_I(n)$ with $\varepsilon_I(n)\to 0$ as $n\to\infty$ such that
\begin{equation}
    \left\lvert \E_{\bmu^n}\left[ I_v^n(\infty) \right] - \frac{1 - c^n(\bmu^n)}{\lambda} \right\rvert \leq \varepsilon_I(n),
\end{equation}
with
\begin{equation}
    \frac{\lambda}{\max_{v \in [n]} \mu^n_v} - \varepsilon_I(n)
    \leq c^n(\bmu^n)
    \leq \frac{\lambda}{\min_{v \in [n]} \mu^n_v} + \varepsilon_I(n),
\end{equation}
for all $n \geq 1$. Then, 
\begin{equation}
\begin{aligned}
    \left\lvert c^n(\bmu^n(t)) - \frac{\lambda}{\mu_v^n(t)} \right\rvert
    &\leq \left\lvert c^n(\bmu^n(t)) - \frac{\lambda}{\max\limits_{v' \in [n]} \mu_{v'}^n(t)} \right\rvert + \left\lvert \frac{\lambda}{\max\limits_{v' \in [n]} \mu_{v'}^n(t)} - \frac{\lambda}{\mu_v^n(t)} \right\rvert \\
    &\leq 2 \left\lvert \frac{\lambda}{\min\limits_{v' \in [n]} \mu_{v'}^n(t)} - \frac{\lambda}{\max\limits_{v' \in [n]} \mu_{v'}^n(t)} \right\rvert + \varepsilon_I(n)
    \leq \frac{2 \lambda \varepsilon_d(n,t)}{\mu_-^2} + \varepsilon_I(n),
\end{aligned}
\end{equation}
for all $v \in [n]$. Therefore,
\begin{equation}
\begin{multlined}
    \left\lvert \E_{\bmu^n(t)}\left[ I_v^n(\infty) \right] - \frac{1 - \lambda / \mu_v^n(t)}{\lambda} \right\rvert 
    \leq \left\lvert \E_{\boldsymbol{\mu}^n(t)}\left[ I_v^n(\infty) \right] - \frac{1 - c^n(\bmu^n)}{\lambda} \right\rvert
    + \frac{1}{\lambda} \left\lvert c^n(\bmu^n) - \frac{\lambda}{\mu_v^n(t)} \right\rvert \\
    \leq \frac{2 \varepsilon_d(n,t)}{\mu_-^2} + \frac{2\varepsilon_I(n)}{\lambda},
\end{multlined}
\end{equation}
and hence,
\begin{equation}
\begin{aligned}
    \frac{\diff}{\diff t} \mu_v^n(t)
    &= - g'(\mu_v^n(t)) - h'(\mu_v^n(t)) \left( \E_{\boldsymbol{\mu}^n(t)}\left[ I_v^n(\infty) \right] + \frac{1}{\mu_v^n(t)} \right) \\
    &= - g'(\mu_v^n(t)) - \frac{h'(\mu_v^n(t))}{\lambda} \pm 2\left( \frac{\varepsilon_d(n,t)}{\mu_-^2} + \frac{\varepsilon_I(n)}{\lambda} \right) \lVert h' \rVert_\infty,
\end{aligned}
\end{equation}
where $x =y\pm z$ means that $x\in [y-z, y+z]$.. Therefore,
\begin{equation}
\begin{aligned}
    \frac{\diff}{\diff t} \left\lvert \lambda g'(\mu_v^n(t)) + h'(\mu_v^n(t)) \right\rvert
    &= \pm \left( \lambda g''(\mu_v^n(t)) + h''(\mu_v^n(t)) \right) \frac{\diff}{\diff t} \mu_v^n(t) \\
    &\leq -\left( \sigma_g + \frac{\sigma_h}{\lambda} \right) \left\lvert \lambda g'(\mu_v^n(t)) + h'(\mu_v^n(t)) \right\rvert + 2\left( \frac{\varepsilon_d(n,t)}{\mu_-^2} + \frac{\varepsilon_I(n)}{\lambda} \right) \lVert h' \rVert_\infty,
\end{aligned}
\end{equation}
for all $n \geq 1$ and $t \geq 0$, which completes the proof of the claim.

\subsection{Proof of Theorem \ref{thm:optimum}} \label{sec:proof_optimum}

First note that
\begin{equation}
\begin{multlined}
    \inf\limits_{\bmu^n \in \mathbb{R}_+^n } \left\{\lambda g\left( \frac{1}{\E[ S^{n}(\infty) ]} \right) + \frac{1}{n} \sum_{v \in [n]} h(\mu_v^{n}) \right\} = \inf\limits_{\overline{\mu}\geq 0} \inf\limits_{ \substack{\bmu^n \in \mathbb{R}_+^n : \\ \frac{1}{n}\sum\limits_{v\in[n]} \mu_v^n = \overline{\mu}} } \left\{\lambda g\left( \frac{1}{\E[ S^{n}(\infty) ]} \right) + \frac{1}{n} \sum_{v \in [n]} h(\mu_v^{n}) \right\} \\
    \geq \inf\limits_{\overline{\mu}\geq 0} \inf\limits_{ \substack{\bmu^n \in \mathbb{R}_+^n : \\ \frac{1}{n}\sum\limits_{v\in[n]} \mu_v^n = \overline{\mu}} } \left\{\lambda g\left( \frac{1}{\E[ S^{n}(\infty) ]} \right) \right\} + \inf\limits_{ \substack{\bmu^n \in \mathbb{R}_+^n : \\ \frac{1}{n}\sum\limits_{v\in[n]} \mu_v^n = \overline{\mu}} } \left\{ \frac{1}{n} \sum_{v \in [n]} h(\mu_v^{n}) \right\} \\
    = \inf\limits_{\overline{\mu}\geq 0} \inf\limits_{ \substack{\bmu^n \in \mathbb{R}_+^n : \\ \frac{1}{n}\sum\limits_{v\in[n]} \mu_v^n = \overline{\mu}} } \left\{\lambda g\left( \frac{1}{\E[ S^{n}(\infty) ]} \right) \right\} + h(\overline{\mu}), \label{eq:aux_opt}
\end{multlined}
\end{equation}
where the last equality is due to the fact that $h$ is nondecreasing and convex. 
To show that the first term is also minimized when $\bmu^n$ is homogeneous in the limit as $n\to\infty$, we argue that the sojourn time is asymptotically minimized in an homogeneous system. To show this, first note that for $n\in \N$ and any fixed service-rate vector $\bmu^n$, a renewal argument implies that
$\mathbb{E}[X^n_v(\infty)] = 1/(1+\mu^n_v\mathbb{E}_\mu[I^n_v(\infty)])$,
where $\mathbb{E}[X^n_v(\infty)]$ is the fraction of time that server $v$ is busy in steady state. 
Moreover, Theorem~\ref{thm:idle} implies that there exists a function $\varepsilon_I(n)$ with $\varepsilon_I(n)\to 0$ as $n\to\infty$ such that
\begin{equation}\label{eq:frac-busy}
    \mathbb{E}[X^n_v(\infty)]=\frac{1}{1+\mu^n_v\mathbb{E}_\mu[I^n_v(\infty)]} \leq \frac{\lambda}{\lambda+\mu^n_v[1 - c^n(\bmu^n)]} + \varepsilon_I(n).
\end{equation}
Furthermore, in steady state we have that
\[ \frac{1}{n} \sum\limits_{v\in[n]} \mu^n_v \mathbb{E}[X^n_v(\infty)] = \lambda_n(\bmu^n), \]
where $\lambda_n(\bmu^n)$ is the per-server effective arrival rate to the system. 
Recall $\overline{\mu}=\frac{1}{n}\sum\limits_{v\in[n]} \mu_v^n$ and note that Equation~\eqref{eq:frac-busy} yields
\begin{equation}
\begin{multlined}
    \lambda_n(\bmu^n) = \frac{1}{n} \sum\limits_{v\in[n]} \mu^n_v \mathbb{E}[X^n_v(\infty)] 
    \leq \frac{1}{n} \sum\limits_{v\in[n]} \mu^n_v \left(\frac{\lambda}{\lambda+\mu^n_v[1 - c^n(\bmu^n)]} + \varepsilon_I(n) \right) \\
    \leq \frac{1}{n} \sum\limits_{v\in[n]} \overline{\mu} \left(\frac{\lambda}{\lambda+\mu^n_v[1 - c^n(\bmu^n)]} + \varepsilon_I(n) \right) 
    \leq \frac{1}{n} \sum\limits_{v\in[n]} \overline{\mu} \, \Big( \mathbb{E}[X^n_v(\infty)] + 2\varepsilon_I(n) \Big),
\end{multlined}
\end{equation}
and thus
\begin{equation}
    \frac{1}{n} \sum\limits_{v\in[n]} \mathbb{E}[X^n_v(\infty)] \geq \frac{\lambda_n(\bmu^n)}{\overline{\mu}} - 2\varepsilon_I(n).
\end{equation}
On the other hand, Little's law implies that
\begin{equation}
    \lambda_n(\bmu^n) \mathbb{E}[S^n(\infty)] = \frac{1}{n} \sum\limits_{v\in[n]} \mathbb{E}[X^n_v(\infty)],
\end{equation}
and Lemmas \ref{thm:supercrit_bound} and \ref{thm:subcrit_bound}
imply that
$
    \lambda_n(\bmu^n) \geq \min\{\lambda, \overline{\mu}\} - \varepsilon_r(n,\overline{\mu}),
$
where $\varepsilon_r(n,\overline{\mu})\in o(1)$ as a function of $n$. Therefore, it follows that
\begin{equation}
    \mathbb{E}[S^n(\infty)] \geq \frac{1}{\overline{\mu}} - \frac{2\varepsilon_I(n)}{\min\{\lambda, \overline{\mu}\} - \varepsilon_r(n,\overline{\mu})},
\end{equation}
and thus
\begin{align}
    \inf\limits_{ \substack{\bmu^n \in \mathbb{R}_+^n : \\ \frac{1}{n}\sum\limits_{v\in[n]} \mu_v^n = \overline{\mu}} } \left\{ g\left( \frac{1}{\E[ S^{n}(\infty) ]} \right) \right\} \geq g(\overline{\mu}) - \varepsilon_T(n,\overline{\mu}),
\end{align} 
where $\varepsilon_T(n,\overline{\mu})$ converges uniformly (over all $\overline{\mu} \geq 0$) to zero, as $n\to\infty$. Combining this with Equation~\eqref{eq:aux_opt} and taking the limit as $n\to\infty$, we obtain
\begin{align}
    \liminf\limits_{n\to\infty} \inf\limits_{\bmu^n \in \mathbb{R}_+^n } \left\{\lambda g\left( \frac{1}{\E[ S^{n}(\infty) ]} \right) + \frac{1}{n} \sum_{v \in [n]} h(\mu_v^{n}) \right\} &\geq \liminf\limits_{n\to\infty} \inf\limits_{\overline{\mu}\geq 0} \Big\{ \lambda g(\overline{\mu}) - \lambda\varepsilon_T(n,\overline{\mu}) + h(\overline{\mu}) \Big\} \\
    &= \inf\limits_{\overline{\mu}\geq 0} \Big\{ \lambda g(\overline{\mu}) + h(\overline{\mu}) \Big\},
\end{align}
where the optimization problem in the right-hand side is now convex, and its infimum is attained at $\mu^*$. Finally, Assumption \ref{ass:opt} (v) implies that $\mu^* > \lambda$.

\section*{Acknowledgements}
The work was partially supported by the NSF grant CIF-2113027.

\def\UrlBreaks{\do\/\do-}
\bibliographystyle{apalike}
\bibliography{main,references-debankur}

\appendix
\renewcommand{\thesection}{\Alph{section}}

\section{Proofs of auxiliary results for Theorem \ref{thm:idle}}
\label{sec:supercrit_bound}
\label{sec:supercrit_idle}
\label{sec:subcrit_transient_limit}
\label{sec:stoch_dom}
\label{sec:subcrit_mixing_time}
\label{sec:subcrit_concentrate}
\label{sec:subcrit_bound}
\label{sec:subcrit_idle}
\label{sec:barphi_bounded}
\begin{proof}[Proof of Lemma \ref{thm:supercrit_bound}]
Note that the evolution of the queue length process can be written as
\begin{equation}
\begin{multlined}
    \sum_{v \in [n]} X_v(t)
    = N_a\left( \int_0^t \lambda n \mathbbm{1}_{\left\{ \sum_{v \in [n]} X_v(s) < n \right\}} \diff s \right) - N_d\left( \int_0^t \sum_{v \in [n]} \mu_v X_v(s) \diff s \right),
\end{multlined}
\end{equation}
where $N_a$ and $N_d$ are independent unit-rate Poisson processes. Let $k := \lfloor (1 - 2 \varepsilon) n \rfloor$ and $Y(t)$ be a Markov process defined as
\begin{equation}
\begin{multlined}
    Y(t) = N_a'\left( \int_0^t \lambda n \mathbbm{1}_{\left\{ Y(s) < k \right\}} \diff s \right) - N_d'\left( \int_0^t \left( \sum_{v \in [n]} \mu_v - 2 \mu_- \varepsilon n \right) \mathbbm{1}_{\left\{ Y(s) > 0 \right\}} \diff s \right),
\end{multlined}
\end{equation}
where $N_a'$ and $N_d'$ are independent unit-rate Poisson processes. Note that if $\sum_{v \in [n]} X_v(t) \leq k$ then $\sum_{v \in [n]} \mu_v X_v(t) = \sum_{v \in [n]} \mu_v - \sum_{v \in [n]} \mu_v (1 - X_v(t)) \leq \sum_{v \in [n]} \mu_v - 2 \mu_- \varepsilon n$ and hence it is not hard to verify that stochastically $Y(t) \leq \sum_{v \in [n]} X_v(t)$ for all $t \geq 0$, given that the inequality holds at $t=0$. 
Moreover, $Y(t)$ is a simple birth-death process and its steady state satisfies
\begin{equation}
    \P\left( Y(\infty) = i \right) = \frac{\rho^i (\rho - 1)}{\rho^{k+1} - 1},\quad i = 0,\ldots, k,
\end{equation}
where
\[ \rho := \frac{\lambda n}{\sum_{v \in [n]} \mu_v - 2 \mu_- \varepsilon n} \geq \frac{\lambda}{\lambda - \mu_- \varepsilon} > 1. \]
Let $l := \lfloor (1 - 3 \varepsilon) n \rfloor$. Then, leveraging the above stochastic ordering, we get
\begin{equation}
\begin{multlined}
    \P\left( \sum_{v \in [n]} X_v(\infty) \leq l \right)
    \leq \P\left( Y(\infty) \leq l \right)
    = \sum_{i = 0}^l \frac{\rho^i (\rho - 1)}{\rho^{k+1} - 1}   = \frac{\rho^{l+1} - 1}{\rho^{k+1} - 1}
    \leq \frac{1}{\rho^{k - l}} \leq \left( 1 - \frac{\mu_- \varepsilon}{\lambda} \right)^{\varepsilon n}.
\end{multlined}
\end{equation}
\end{proof}
To prove Lemma \ref{thm:supercrit_idle}, we will need the following technical result:
\begin{lemma}
\label{lem:lyapunov_eq}
Fix any $v \in [n]$. Then,
\begin{equation}
\begin{aligned}
    \E\left[ \frac{\lambda n (1 - X_v(\infty))}{\sum_{v' \in [n]} (1 - X_{v'}(\infty))} \right] &= \E\left[ \mu_v X_v(\infty) \right], \\
    \E\left[ \mu_v X_v(\infty) I_v(\infty) \right] &= \E\left[ 1 - X_v(\infty) \right], \\
    \E\left[ \frac{\lambda n (1 - X_v(\infty)) I_v(\infty)}{\sum_{v' \in [n]} (1 - X_{v'}(\infty))} \right] &= \E\left[ 1 - X_v(\infty) \right].
\end{aligned}
\end{equation}
\end{lemma}

\begin{proof}
Consider a system initiated at the steady state, i.e., $\boldsymbol{X}(0) = \boldsymbol{X}(\infty)$. Then,
\begin{equation}
    \frac{\diff}{\diff t} \E\left[ X_v(t) \right]
    = \E\left[ \frac{\lambda n (1 - X_v(t))}{\sum_{v' \in [n]} (1 - X_{v'}(t))} - \mu_v X_v(t) \right]
    = 0,
\end{equation}
and
\begin{equation}
    \frac{\diff}{\diff t} \E\left[ I_v(t) \right]
    = \E\left[ 1 - X_v(t) - \mu_v X_v(t) I_v(t) \right]
    = 0,
\end{equation}
and
\begin{equation}
    \frac{\diff}{\diff t} \E\left[ I_v(t) (1 - X_v(t)) \right]
    = \E\left[ 1 - X_v(t) - \frac{\lambda n (1 - X_v(t)) I_v(t)}{\sum_{v' \in [n]} (1 - X_{v'}(t))} \right]
    = 0,
\end{equation}
which show the three equalities in the lemma, respectively.
\end{proof}

\begin{proof}[Proof of Lemma \ref{thm:supercrit_idle}]
Let $E := \big\{ \sum_{v \in [n]} (1 - X_v(\infty)) \leq 3 \varepsilon n \big\}$. Then,
\begin{equation}\label{eq:supercrit_idle_1}
\begin{multlined}
    \E\left[ (1 - X_v(\infty)) I_v(\infty) \right]    
    = \E\left[ \frac{1}{\lambda} \sum_{v' \in [n]} (1 - X_{v'}(\infty)) \frac{\lambda (1 - X_v(\infty)) I_v(\infty)}{\sum_{v' \in [n]} (1 - X_{v'}(\infty))} \mathbbm{1}_E \right] \\  + \E\left[ (1 - X_v(\infty)) I_v(\infty) \mathbbm{1}_{E^c} \right]
    \leq \E\left[ \frac{3 \varepsilon}{\lambda} \frac{\lambda n (1 - X_v(\infty)) I_v(\infty)}{\sum_{v' \in [n]} (1 - X_{v'}(\infty))} \right]
    + \sqrt{\E\left[ I_v(\infty)^2 \right] \P(E^c)}  \\
    \leq \frac{3 \varepsilon}{\lambda}
    + \frac{\sqrt{2}}{\lambda} \left( 1 - \frac{\mu_- \varepsilon}{\lambda} \right)^{\varepsilon n / 2},
\end{multlined}
\end{equation}
where the first inequality follows by Cauchy-Schwartz and the second inequality follows by Lemma \ref{lem:lyapunov_eq}, the fact that $I_v(\infty)$ is stochastically dominated by an $\text{Exp}(\lambda)$ random variable and Lemma~\ref{thm:supercrit_bound}. Also,
\begin{equation}
\begin{multlined}
\label{eq:supercrit_idle_2}
    \E\left[ X_v(\infty) I_v(\infty) \right]
    = \E\left[ \frac{1 - X_v(\infty)}{\mu_v} \right]
    \leq \frac{3 \varepsilon}{\lambda}
    + \frac{1}{\mu_-} \left( 1 - \frac{\mu_- \varepsilon}{\lambda} \right)^{\varepsilon n / 2},
\end{multlined}
\end{equation}
where the equality follows from Lemma \ref{lem:lyapunov_eq} and the inequality follows along similar lines as Equation \eqref{eq:supercrit_idle_1}. Therefore, the proof follows by adding equations \eqref{eq:supercrit_idle_1} and \eqref{eq:supercrit_idle_2}.
\end{proof}

The following algebraic lemma is used in the proof of Lemma \ref{thm:subcrit_transient_limit}.
\begin{lemma} \label{lem:inv_lipschitz}
Let $0 \leq y_1 \leq x_1$ and $0 \leq y_2 \leq x_2$. Then,
\begin{equation}
    \left\lvert \frac{y_1}{x_1} - \frac{y_2}{x_2} \right\rvert
    \leq \frac{1}{\min\{ x_1, x_2 \}} \left( \left\lvert x_1 - x_2 \right\rvert + \left\lvert y_1 - y_2 \right\rvert \right).
\end{equation}
\end{lemma}

\begin{proof}
Let $x_1 \leq x_2$ without loss of generality. Then, by the mean value theorem,
    $\frac{1}{x_1} - \frac{1}{x_2} = - \frac{x_1 - x_2}{\xi^2},$
where $\xi \in [x_1, x_2]$. Therefore,
\begin{equation}
\begin{multlined}
    \left\lvert \frac{y_1}{x_1} - \frac{y_2}{x_2} \right\rvert
    \leq \left\lvert \frac{y_1}{x_1} - \frac{y_1}{x_2} \right\rvert + \left\lvert \frac{y_1}{x_2} - \frac{y_2}{x_2} \right\rvert  \leq \frac{y_1}{x_1^2} \left\lvert x_1 - x_2 \right\rvert + \frac{1}{x_2} \left\lvert y_1 - y_2 \right\rvert
    \leq \frac{1}{x_1} \left( \left\lvert x_1 - x_2 \right\rvert + \left\lvert y_1 - y_2 \right\rvert \right),
\end{multlined}
\end{equation}
which completes the proof.
\end{proof}

\begin{proof}[Proof of Lemma \ref{thm:subcrit_transient_limit}]
Let $f: \R_+ \to [0, 1]$ be any function. The queue length process evolves as
\begin{equation}
\begin{multlined}
    \frac{1}{n} \sum_{v \in [n]} f(\mu_v) X_v(t)
    = \frac{1}{n} \sum_{v \in [n]} f(\mu_v) X_v(0)  + \frac{1}{n} \sum_{v \in [n]} f(\mu_v) N_v^{(1)}\left( \int\limits_0^t \frac{\lambda n (1 - X_v(s))}{\sum\limits_{v' \in [n]} (1 - X_{v'}(s))} \diff s \right) \\ 
    - \frac{1}{n} \sum_{v \in [n]} f(\mu_v) N_v^{(2)}\left( \int\limits_0^t \mu_v X_v(s) \diff s \right), \label{eq:queue_evolution}
\end{multlined}
\end{equation}
where $N_v^{(1)}$ and $N_v^{(2)}$ are independent unit-rate Poisson processes for $v \in [n]$. We rewrite the second and third terms on the right-hand side of Equation \eqref{eq:queue_evolution} to get
\begin{equation}
\begin{multlined}
\label{eq:martingale_decomposition}
    \frac{1}{n} \sum_{v \in [n]} f(\mu_v) X_v(t)
    = \frac{1}{n} \sum_{v \in [n]} f(\mu_v) X_v(0) + \frac{1}{n} \sum_{v \in [n]} f(\mu_v) \int\limits_0^t \frac{\lambda n (1 - X_v(s))}{\sum_{v' \in [n]} (1 - X_{v'}(s))} \diff s + M_f^{(1)}(t) \\
    - \frac{1}{n} \sum_{v \in [n]} f(\mu_v) \int\limits_0^t \mu_v X_v(s) \diff s - M_f^{(2)}(t),
\end{multlined}
\end{equation}
where
\begin{equation}
\begin{aligned}
    M_f^{(1)}(t) &:= \frac{1}{n} \sum_{v \in [n]} f(\mu_v) M_v^{(1)}\left( \int\limits_0^t \frac{\lambda n (1 - X_v(s))}{\sum_{v' \in [n]} (1 - X_{v'}(s))} \diff s \right), \\
    M_f^{(2)}(t) &:= \frac{1}{n} \sum_{v \in [n]} f(\mu_v) M_v^{(2)}\left( \int\limits_0^t \mu_v X_v(s) \diff s \right),
\end{aligned}
\end{equation}
and $M_v^{(i)}(t) := N_v^{(i)}(t) - t$ for $i = 1, 2$ and $v \in [n]$. Let $\mathcal{F}_t$ be the natural filtration for $\boldsymbol{X}(t)$. Then, it is easy to check that $M_f^{(1)}(t)$ and $M_f^{(2)}(t)$ are square-integrable martingales with respect to $\mathcal{F}_t$. Moreover, we have
\begin{equation}
\begin{multlined}
\label{eq:quadvar_1}
    \E\left[ \left\langle M_f^{(1)}, M_f^{(1)} \right\rangle(t) \right]    = \frac{1}{n^2} \sum_{v \in [n]} f(\mu_v)^2 \E\left[ \left\langle N_v^{(1)}, N_v^{(1)} \right\rangle\left( \int\limits_0^t \mu_v X_v(s) \diff s \right) \right] \\
    = \frac{1}{n^2} \sum_{v \in [n]} f(\mu_v)^2 \E\left[ \int\limits_0^t \mu_v X_v(s) \diff s \right]
    \leq \frac{\mu_+ t}{n},
\end{multlined}
\end{equation}
and
\begin{equation}
\begin{multlined}
\label{eq:quadvar_2}
    \E\left[ \left\langle M_f^{(2)}, M_f^{(2)} \right\rangle(t) \right]   = \frac{1}{n^2} \sum_{v \in [n]} f(\mu_v)^2 \E\left[ \left\langle N_v^{(2)}, N_v^{(2)} \right\rangle\left( \int\limits_0^t \frac{\lambda n (1 - X_v(s))}{\sum_{v' \in [n]} (1 - X_{v'}(s))} \diff s \right) \right] \\
    = \frac{1}{n^2} \sum_{v \in [n]} f(\mu_v)^2 \E\left[ \int\limits_0^t \frac{\lambda n (1 - X_v(s))}{\sum_{v' \in [n]} (1 - X_{v'}(s))} \diff s \right]
    \leq \frac{\lambda t}{n}.
\end{multlined}
\end{equation}
Now, we claim the following:
\begin{claim}\label{claim:int-phi0}
For all $t \geq 0$, we have $\int_0^\infty \bar{\phi}_t(\diff x) \leq 1 - \delta / \mu_+$. 
\end{claim}
\begin{claimproof}
Note that by definition, $\int_0^\infty \bar{\phi}_t(\diff x)$ is differentiable as a function of $t$. Then, if $t$ is any time such that $\int_0^\infty \bar{\phi}_t(\diff x) = 1 - \delta / \mu_+$, we have
\begin{equation}
\begin{multlined}
    \frac{\diff}{\diff t} \int_0^\infty \bar{\phi}_t(\diff x)
    = \frac{\lambda \left( \int_0^\infty \Phi(\diff x) - \int_0^\infty \bar{\phi}_t(\diff x) \right)}{1 - \int_0^\infty \bar{\phi}_t(\diff x)} - \int_0^\infty x \bar{\phi}_t(\diff x) \\
    = \left( \lambda - \int_0^\infty x \Phi(\diff x) \right) + \left( \int_0^\infty x \Phi(\diff x) - \int_0^\infty x \bar{\phi}_t(\diff x) \right) 
    \leq -\delta + \mu_+ \left( 1 - \int_0^\infty \bar{\phi}_t(\diff x) \right)
    = 0.
\end{multlined}
\end{equation}
Thus, the claim follows.    
\end{claimproof}
Next, let us define
    $d_f(t) := \left\lvert \frac{1}{n} \sum_{v \in [n]} f(\mu_v) X_v(t)- \int_0^\infty f(x) \bar{\phi}_t(\diff x) \right\rvert$.
Then, Lemma \ref{lem:inv_lipschitz}, Equation \eqref{eq:martingale_decomposition}, and Claim~\ref{claim:int-phi0} imply that 
\begin{equation}
\begin{split}
    d_f(t)
    &\leq \int\limits_0^t \lambda \left\lvert \frac{\frac{1}{n} \sum_{v \in [n]} f(\mu_v) (1 - X_v(s))}{\frac{1}{n} \sum_{v \in [n]} (1 - X_v(s))} - \frac{\int_0^\infty f(x) \Phi(\diff x) - \int_0^\infty f(x) \bar{\phi}_s(\diff x)}{1 - \int_0^\infty \bar{\phi}_s(\diff x)} \right\rvert \diff s \\
    &\hspace{3cm}+ \int\limits_0^t \left\lvert \frac{1}{n} \sum_{v \in [n]} \mu_v f(\mu_v) X_v(s) - \int\limits_0^\infty x f(x) \bar{\phi}_s(\diff x) \right\rvert \diff s + \left\lvert M_f^{(1)}(t) \right\rvert + \left\lvert M_f^{(2)}(t) \right\rvert \\
    &\leq \int\limits_0^t \lambda \left( \frac{2 \mu_+ (d_{\boldsymbol{1}}(s) + d_f(s))}{\delta} + \mathbbm{1}_{\left\{ \frac{1}{n} \sum_{v \in [n]} (1 - X_v(s)) < \frac{\delta}{2 \mu_+} \right\}} \right) \diff s     + \int\limits_0^t \mu_+ d_g(s) \diff s \\ &\hspace{3cm} + \left\lvert M_f^{(1)}(t) \right\rvert + \left\lvert M_f^{(2)}(t) \right\rvert \\
    &\leq \int\limits_0^t \lambda \left( \frac{2 \mu_+ (d_{\boldsymbol{1}}(s) + d_f(s))}{\delta} + \mathbbm{1}_{\left\{ d_{\boldsymbol{1}}(s) > \frac{\delta}{2 \mu_+} \right\}} \right) + \mu_+ d_g(s) \diff s + \left\lvert M_f^{(1)}(t) \right\rvert + \left\lvert M_f^{(2)}(t) \right\rvert \\
    &\leq \int\limits_0^t \lambda \left( \frac{2 \mu_+ (2 d_{\boldsymbol{1}}(s) + d_f(s))}{\delta} \right) + \mu_+ d_g(s) \diff s + \left\lvert M_f^{(1)}(t) \right\rvert + \left\lvert M_f^{(2)}(t) \right\rvert,
\end{split}
\end{equation}
where $g(x) := x f(x) / \mu_+$ for $x \in \R_+$. Then,
\begin{equation}\label{eq:dft-bdd}
\begin{multlined}
    \sup_{f: \R_+ \to [0, 1]} \E\left[ \sup_{t \in [0, T]} d_f(t) \right]
    \leq \int\limits_0^T \left( \frac{6 \lambda \mu_+ }{\delta} + \mu_+ \right) \sup_{f: \R_+ \to [0, 1]} \E\left[ \sup_{s \in [0, t]} d_f(s) \right] \diff t \\ 
    + \sup_{f: \R_+ \to [0, 1]} \E\left[ \sup_{t \in [0, T]} \left\lvert M_f^{(1)}(t) \right\rvert \right] + \sup_{f: \R_+ \to [0, 1]} \E\left[ \sup_{t \in [0, T]} \left\lvert M_f^{(2)}(t) \right\rvert \right].
\end{multlined}
\end{equation}
Note that, by Doob's maximal inequality and equations \eqref{eq:quadvar_1} and \eqref{eq:quadvar_2}, we have
\begin{equation}\label{eq:m1f-bdd}
    \E\left[ \sup_{t \in [0, T]} \left\lvert M_f^{(1)}(t) \right\rvert \right]
    \leq \sqrt{\E\left[ \sup_{t \in [0, T]} M_f^{(1)}(t)^2 \right]}
    \leq \sqrt{4 \E\left[ M_f^{(1)}(T)^2 \right]}
    \leq \sqrt{\frac{4 \mu_+ T}{n}},
\end{equation}
and
\begin{equation}\label{eq:m2f-bdd}
    \E\left[ \sup_{t \in [0, T]} \left\lvert M_f^{(2)}(t) \right\rvert \right]
    \leq \sqrt{\E\left[ \sup_{t \in [0, T]} M_f^{(2)}(t)^2 \right]}
    \leq \sqrt{4 \E\left[ M_f^{(2)}(T)^2 \right]}
    \leq \sqrt{\frac{4 \lambda T}{n}}.
\end{equation}
Therefore, combining equations \eqref{eq:dft-bdd}, \eqref{eq:m1f-bdd} and \eqref{eq:m2f-bdd}, and then using Gr\"{o}nwall's inequality completes the proof of Lemma \ref{thm:subcrit_transient_limit}.
\end{proof}


\begin{proof}[Proof of Lemma \ref{lem:stoch_dom}]
We will construct a joint probability space such that the arrival and potential departure epochs are coupled in the two systems. Then we will prove the statement by induction on the coupled arrival and departure epochs.

Let $t \geq 0$ be an arrival epoch and assume that $\boldsymbol{X}^{(1)}(t-) \leq \boldsymbol{X}^{(2)}(t-)$ before the arrival. 
Define $S^{(i)} := \{ v \in [n] : X_v^{(i)}(t-) = 0 \}$ to be the set of idle servers in system $i$ for $i = 1, 2$ and note that $S^{(2)} \subseteq S^{(1)}$ by the induction hypothesis. If $S^{(1)} = \varnothing$ or $S^{(2)} = \varnothing$, then $\boldsymbol{X}^{(1)}(t) \leq \boldsymbol{X}^{(2)}(t)$ is trivially maintained. 
Otherwise, we assign the idle servers in $S^{(2)}$ an index from $1$ to $\lvert S^{(2)} \rvert$ and we assign the idle servers in $S^{(1)} \setminus S^{(2)}$ an index from $\lvert S^{(2)} \rvert + 1$ to $\lvert S^{(1)} \rvert$ (if any). 
Let $U$ be a uniform $[0, 1]$ random variable, independent across time epochs and shared between the two systems. 
Now the task assignment decision (which depends on $U$ in that epoch) is taken in the two systems as follows:
\begin{itemize}
    \item In system 1, we assign the task to the $j$-th idle server for $j = 1, \dots, \lvert S^{(1)} \rvert$ if and only if
    $U \in \left[ \frac{j-1}{\lvert S^{(1)} \rvert}, \frac{j}{\lvert S^{(1)} \rvert} \right).$

    \item In system 2, we assign the task to the $j$-th idle server for $j = 1, \dots, \lvert S^{(2)} \rvert$ if and only if
\begin{equation}
    U \in \left[ \frac{j-1}{\lvert S^{(1)} \rvert}, \frac{j}{\lvert S^{(1)} \rvert} \right) \cup \left[ \frac{\lvert S^{(2)} \rvert}{\lvert S^{(1)} \rvert} + (j-1) \left( \frac{1}{\lvert S^{(2)} \rvert} - \frac{1}{\lvert S^{(1)} \rvert} \right), \frac{\lvert S^{(2)} \rvert}{\lvert S^{(1)} \rvert} + j \left( \frac{1}{\lvert S^{(2)} \rvert} - \frac{1}{\lvert S^{(1)} \rvert} \right) \right).
\end{equation}
\end{itemize}
Note that the probability that any idle server is picked equals $1 / \lvert S^{(1)} \rvert$ in system 1 and $1 / \lvert S^{(2)} \rvert$ in system 2 as required. If $U < \lvert S^{(2)} \rvert / \lvert S^{(1)} \rvert$, then the task is routed to the same server in both systems and hence $\boldsymbol{X}^{(1)}(t) \leq \boldsymbol{X}^{(2)}(t)$ is trivially maintained. 
Now, if $U \geq \lvert S^{(2)} \rvert / \lvert S^{(1)} \rvert$, then the task is routed to a server $v \in S^{(1)} \setminus S^{(2)}$ in system 1. This means that server $v$ is already busy in system 2 and hence $\boldsymbol{X}^{(1)}(t) \leq \boldsymbol{X}^{(2)}(t)$ is maintained. The arrival in system 2 only increases the queue lengths in system 2 and hence does not invalidate the inequality.

We also synchronize the potential departure epochs in server $v$ in the two systems for all $v\in [n]$. Let $t \geq 0$ be such a potential departure epoch at server $v$ and assume that $\boldsymbol{X}^{(1)}(t-) \leq \boldsymbol{X}^{(2)}(t-)$ before the departure. After the departure, $X_v^{(1)}(t) = X_v^{(2)}(t) = 0$ and hence the inequality $\boldsymbol{X}^{(1)}(t-) \leq \boldsymbol{X}^{(2)}(t-)$ is trivially maintained.
\end{proof}


\begin{proof}[Proof of Lemma \ref{thm:subcrit_mixing_time}]
Note that the occupancy processes in the two systems evolve as
\begin{equation}
\begin{multlined}
    \sum_{v \in [n]} X_v^{(i)}(t)
    = \sum_{v \in [n]} X_v^{(i)}(0)  + N_a^{(i)}\left( \int\limits_0^t \lambda n \mathbbm{1}_{\left\{ \sum_{v \in [n]} X_v^{(i)}(s) < n \right\}} \diff s \right)   - N_d^{(i)}\left( \int\limits_0^t \sum_{v \in [n]} \mu_v X_v^{(i)}(s) \diff s \right),
\end{multlined}
\end{equation}
where $N_a^{(i)}$ and $N_d^{(i)}$ are independent unit-rate Poisson processes for $i = 1, 2$. Therefore, using Fubini's theorem,
\begin{equation}
\begin{multlined}
    \E\left[ \sum_{v \in [n]} X_v^{(i)}(t) \right]
    = \E\left[ \sum_{v \in [n]} X_v^{(i)}(0) \right]   + \int\limits_0^t \lambda n \P\left( \sum_{v \in [n]} X_v^{(i)}(s) < n \right) \diff s   - \int\limits_0^t \E\left[ \sum_{v \in [n]} \mu_v X_v^{(i)}(s) \right] \diff s,
\end{multlined}
\end{equation}
for $i = 1, 2$. We let the two processes be defined on the joint probability space defined in Lemma \ref{lem:stoch_dom}. This implies that almost surely, $\boldsymbol{X}^{(1)}(t) \leq \boldsymbol{X}^{(2)}(t)$ for all $t \geq 0$ and hence, 
\[ \P\left( \sum_{v \in [n]} X_v^{(2)}(t) < n \right) \leq \P\left( \sum_{v \in [n]} X_v^{(1)}(t) < n \right) \]
for all $t \geq 0$. Further, note that the coupling also ensures that
$\sum_{v \in [n]} \left( X_v^{(2)}(t) - X_v^{(1)}(t) \right) $
is nonincreasing and it decreases by 1 whenever there is a departure from system 2 but not from system 1.
Thus,
\begin{equation}
\begin{multlined}
    \E\left[ \sum_{v \in [n]} \left( X_v^{(2)}(t) - X_v^{(1)}(t) \right) \right]  \leq \E\left[ \sum_{v \in [n]} \left( X_v^{(2)}(0) - X_v^{(1)}(0) \right) \right] \\
    - \int\limits_0^t \E\left[ \sum_{v \in [n]} \mu_v \left( X_v^{(2)}(s) - X_v^{(1)}(s) \right) \right] \diff s
    \leq n
    - \int\limits_0^t \mu_- \E\left[ \sum_{v \in [n]} \left( X_v^{(2)}(s) - X_v^{(1)}(s) \right) \right] \diff s.
\end{multlined}
\end{equation}
Therefore, we have
$\E\left[ \sum_{v \in [n]} \left( X_v^{(2)}(t) - X_v^{(1)}(t) \right) \right] \leq n \exp(-\mu_- t).$
Dividing by $n$, the result follows.
\end{proof}


\begin{proof}[Proof of Lemma \ref{thm:subcrit_concentrate}]
We consider two copies of the queue length process such that $\boldsymbol{X}^{(1)}(0) = 0$ and $\boldsymbol{X}^{(2)}(0) = \boldsymbol{X}^{(2)}(\infty)$ and defined on the joint probability space of Lemma~\ref{lem:stoch_dom}. Then, Lemma~\ref{thm:subcrit_mixing_time} implies that
$
    \E\left[ \frac{1}{n} \sum_{v \in [n]} \left\lvert X_v^{(2)}(t) - X_v^{(1)}(t) \right\rvert \right]
    \leq \exp\left( -\mu_- t \right).
$
Also, Lemma \ref{thm:subcrit_transient_limit} implies that
\begin{equation}
\begin{multlined}
    \E\left[ \left\lvert \frac{1}{n} \sum_{v \in [n]} X_v^{(1)}(t) - \int\limits_0^\infty \bar{\phi}_t(\diff x) \right\rvert \right]
    \leq \sqrt{\frac{8 (\mu_+ + \lambda) t}{n}} \exp\left( \left( \frac{6 \lambda \mu_+}{\delta} + \mu_+ \right) t \right) \\ \leq \sqrt{\frac{8 (\mu_+ + \lambda) t}{n}} \exp\left( \left( \frac{6 \lambda \mu_+}{\varepsilon \mu_-} + \mu_+ \right) t \right).
\end{multlined}
\end{equation}
Let $t = \log(n) / \alpha$. Then,
\begin{equation}
\begin{multlined}
    \E\left[ \left\lvert \frac{1}{n} \sum_{v \in [n]} X_v^{(2)}(t) - c \right\rvert \right]
    \leq \exp\left( -\mu_- t \right) + \sqrt{\frac{8 (\mu_+ + \lambda) t}{n}} \exp\left( \left( \frac{6 \lambda \mu_+}{\varepsilon \mu_-} + \mu_+ \right) t \right) \\  = \frac{1 + \sqrt{\frac{8 (\mu_+ + \lambda) \log(n)}{\alpha}}}{n^{\frac{\mu_-}{\alpha}}}.
\end{multlined}
\end{equation}
Since $\boldsymbol{X}^{(2)}(0) = \boldsymbol{X}^{(2)}(\infty)$, we have $\boldsymbol{X}^{(2)}(t) = \boldsymbol{X}^{(2)}(\infty)$, and the result follows.
\end{proof}

\begin{proof}[Proof of Lemma \ref{thm:subcrit_bound}]
As before, note that the occupancy process, starting from the empty state, evolves as
\begin{equation}
    \sum_{v \in [n]} X_v(t)
    = N_a\left( \int_0^t \lambda n \mathbbm{1}_{\left\{ \sum_{v \in [n]} X_v(s) < n \right\}} \diff s \right) - N_d\left( \int_0^t \sum_{v \in [n]} \mu_v X_v(s) \diff s \right)
\end{equation}
where $N_a$ and $N_d$ are independent unit-rate Poisson processes. Let $k := \lfloor (1 - 2 \varepsilon) n \rfloor$ and let $Y(t)$ be a Markov process defined as
\begin{equation}
\begin{multlined}
    Y(t) = N_a'\left( \int_0^t \lambda n \mathbbm{1}_{\left\{ Y(s) < n \right\}} \diff s \right) - N_d'\left( \int_0^t \left( \sum_{v \in [n]} \mu_v - 2 \mu_+ \varepsilon n \right) \mathbbm{1}_{\left\{ Y(s) > k \right\}} \diff s \right),
\end{multlined}
\end{equation}
where $N_a'$ and $N_d'$ are independent unit-rate Poisson processes. Note that if $\sum_{v \in [n]} X_v(t) > k$ then
\[ \sum_{v \in [n]} \mu_v X_v(t) = \sum_{v \in [n]} \mu_v - \sum_{v \in [n]} \mu_v (1 - X_v(t)) \geq \sum_{v \in [n]} \mu_v - 2 \mu_+ \varepsilon n \]
and hence standard coupling can be constructed so that almost surely, $Y(t) \geq \sum_{v \in [n]} X_v(t)$ for all $t \geq 0$, provided it is satisfied at $t=0$. 
Moreover, $Y(t)$ is a simple birth-death process and its steady state satisfies 
\begin{equation}
    \P\left( Y(\infty) = i \right) = \frac{\rho^i (1 - \rho)}{\rho^{n+1} - \rho^k},\qquad i = k, k+1, \ldots, n,
\end{equation}
where
\[ \rho := \lambda n \left( \sum_{v \in [n]} \mu_v - 2 \mu_+ \varepsilon n \right)^{-1} \leq \frac{\lambda}{\lambda + \mu_+ \varepsilon} < 1. \]
Let $l := \lfloor (1 - \varepsilon) n \rfloor$. Then,
\begin{equation}
\begin{multlined}
    \P\left( \sum_{v \in [n]} X_v(\infty) \geq l \right)
    \leq \P\left( Y(\infty) \geq l \right)
    = \sum_{i = l}^n \frac{\rho^i (1 - \rho)}{\rho^k - \rho^{n+1}}    = \frac{\rho^l - \rho^{n+1}}{\rho^k - \rho^{n+1}}
    \leq \rho^{l - k} \leq \left( 1 - \frac{\mu_+ \varepsilon}{\lambda + \mu_+ \varepsilon} \right)^{\varepsilon n}.
\end{multlined}
\end{equation}
\end{proof}


\begin{proof}[Proof of Lemma \ref{thm:subcrit_idle}]
We know that
\begin{equation}
\begin{split}
\label{eq:subcrit_idle_1}
    \E\Big[ &(1 - X_v(\infty)) I_v(\infty) \Big]
    = \E\left[ \frac{1}{n} \sum_{v' \in [n]} (1 - X_{v'}(\infty)) \frac{n (1 - X_v(\infty)) I_v(\infty)}{\sum_{v' \in [n]} (1 - X_{v'}(\infty))} \right] \\
   & = \E\left[ \frac{1 - c}{\lambda} \frac{\lambda n (1 - X_v(\infty)) I_v(\infty)}{\sum_{v' \in [n]} (1 - X_{v'}(\infty))} \right]    \pm \E\left[ \left\lvert \frac{1}{n} \sum_{v' \in [n]} (1 - X_{v'}(\infty)) - (1 - c) \right\rvert \frac{n (1 - X_v(\infty)) I_v(\infty)}{\sum_{v' \in [n]} (1 - X_{v'}(\infty))} \right].
\end{split}
\end{equation}
By Lemma \ref{lem:lyapunov_eq}, the first term on the right-hand side of Equation \eqref{eq:subcrit_idle_1} satisfies
\begin{equation}
\begin{multlined}
    \E\left[ \frac{1 - c}{\lambda} \frac{\lambda n (1 - X_v(\infty)) I_v(\infty)}{\sum_{v' \in [n]} (1 - X_{v'}(\infty))} \right]
    = \E\left[ \frac{(1 - c) (1 - X_v(t))}{\lambda} \right],
\end{multlined}
\end{equation}
Now let
\[ E_1 := \left\{ \sum_{v' \in [n]} \big(1 - X_{v'}(\infty)\big) \geq \frac{\delta n}{3 \mu_+} \right\} \qquad \text{and} \qquad  E_2 := \left\{ \left\lvert \frac{1}{n} \sum_{v' \in [n]} X_{v'} - c \right\rvert \leq \varepsilon \right\}. \]
The second term on the right-hand side of Equation~\eqref{eq:subcrit_idle_1} satisfies
\begin{equation*}
\begin{split}
\label{eq:subcrit_idle_3}
    \E\Big[ \Big\lvert \frac{1}{n} & \sum_{v' \in [n]} (1 - X_{v'}(\infty)) - (1 - c) \Big\rvert \frac{n (1 - X_v(\infty)) I_v(\infty)}{\sum_{v' \in [n]} (1 - X_{v'}(\infty))} \Big] \\
    &= \E\left[ \left\lvert \frac{1}{n} \sum_{v' \in [n]} X_{v'}(\infty)) - c \right\rvert \frac{n (1 - X_v(\infty)) I_v(\infty)}{\sum_{v' \in [n]} (1 - X_{v'}(\infty))} \left( \mathbbm{1}_{E_1} + \mathbbm{1}_{E_1^c} \right) \right] \\
    &\leq \E\left[ \left\lvert \frac{1}{n} \sum_{v' \in [n]} X_{v'}(\infty) - c \right\rvert \frac{n (1 - X_v(\infty)) I_v(\infty)}{\sum_{v' \in [n]} (1 - X_{v'}(\infty))} (\mathbbm{1}_{E_2} + \mathbbm{1}_{E_2^c}) \mathbbm{1}_{E_1} \right]   + \sqrt{\E\left[ n^2 I_v(\infty)^2 \right] \P(E_1^c)} \\
    &\leq \E\left[ \frac{\varepsilon}{\lambda} \frac{\lambda n (1 - X_v(\infty)) I_v(\infty)}{\sum_{v' \in [n]} (1 - X_{v'}(\infty))} \right]
    + \sqrt{ \E\left[ \frac{I_v(\infty)^2}{(\delta / (3 \mu_+))^2} \right] \P(E_2^c)}     + \frac{\sqrt{2} n}{\lambda} \left( 1 - \frac{\delta / 3}{\lambda + \delta / 3} \right)^{\delta n / (6 \mu_+)} \\
    &\leq \frac{\varepsilon}{\lambda}
    + \frac{3 \sqrt{2} \mu_+}{\delta \lambda} \sqrt{\frac{1 + \sqrt{\frac{8 (\mu_+ + \lambda) \log(n)}{\alpha}}}{n^{\frac{\mu_-}{\alpha}}}}
    + \frac{\sqrt{2} n}{\lambda} \left( 1 - \frac{\delta}{3 \lambda + \delta} \right)^{\delta n / (6 \mu_+)},
\end{split}
\end{equation*}
where the first inequality follows by Cauchy-Schwartz, the second inequality follows by Cauchy-Schwartz, the fact that $I_v(\infty)$ is stochastically dominated by an $\text{Exp}(\lambda)$, and Lemma~\ref{thm:subcrit_bound}, and the third inequality follows by Lemmas \ref{lem:lyapunov_eq} and \ref{thm:subcrit_concentrate}. Also,
\begin{equation}
\begin{multlined}
\label{eq:subcrit_idle_2}
    \E\left[ X_v(\infty) I_v(\infty) \right]
    = \E\left[ \frac{(1 - X_v(\infty))}{\mu_v} \right]     = \E\left[ \frac{1}{\mu_v n} \sum_{v' \in [n]} (1 - X_{v'}(\infty)) \frac{n (1 - X_v(\infty))}{\sum_{v' \in [n]} (1 - X_{v'}(\infty))} \right] \\
    = \E\left[ \frac{1 - c}{\lambda \mu_v} \frac{\lambda n (1 - X_v(\infty))}{\sum_{v' \in [n]} (1 - X_{v'}(\infty))} \right] 
    \pm \E\left[ \frac{1}{\mu_v} \left\lvert \frac{1}{n} \sum_{v' \in [n]} (1 - X_{v'}(\infty)) - (1 - c) \right\rvert \frac{n (1 - X_v(\infty))}{\sum_{v' \in [n]} (1 - X_{v'}(\infty))} \right],
\end{multlined}
\end{equation}
where the first equality follows by Lemma \ref{lem:lyapunov_eq}. The first term on the right-hand side of Equation \eqref{eq:subcrit_idle_2} satisfies
\begin{equation}
    \E\left[ \frac{1 - c}{\lambda \mu_v} \frac{\lambda n (1 - X_v(\infty))}{\sum_{v' \in [n]} (1 - X_{v'}(\infty))} \right]
    = \E\left[ \frac{(1 - c) X_v(\infty)}{\lambda} \right],
\end{equation}
by Lemma \ref{lem:lyapunov_eq}. The second term on the right-hand side of Equation \eqref{eq:subcrit_idle_2} satisfies
\begin{equation}
\begin{multlined}
    \E\left[ \frac{1}{\mu_v} \left\lvert \frac{1}{n} \sum_{v' \in [n]} (1 - X_{v'}(\infty)) - (1 - c) \right\rvert \frac{n (1 - X_v(\infty))}{\sum_{v' \in [n]} (1 - X_{v'}(\infty))} \right] \\
    \leq \frac{\varepsilon}{\lambda}
    + \frac{3 \mu_+}{\delta \mu_-} \sqrt{\frac{1 + \sqrt{\frac{8 (\mu_+ + \lambda) \log(n)}{\alpha}}}{n^{\frac{\mu_-}{\alpha}}}}
    + \frac{n}{\mu_-} \left( 1 - \frac{\delta}{3 \lambda + \delta} \right)^{\delta n / (6 \mu_+)},
\end{multlined}
\end{equation}
which follows along similar lines as Equation \eqref{eq:subcrit_idle_3}. Therefore, the proof follows by adding equations \eqref{eq:subcrit_idle_1} and \eqref{eq:subcrit_idle_2} and the fact that $\varepsilon \leq \delta / \mu_-$.
\end{proof}

\begin{proof}[Proof of Lemma \ref{lem:barphi_bounded}]
We define
\begin{equation} \label{eq:def_y}
\begin{aligned}
    y_1(t) &:= \int\limits_0^t \left( \lambda - \max_{v \in [n]} \mu_v y_1(s) \right) \diff s, \qquad \text{and} \qquad y_2(t) := \int\limits_0^t \left( \lambda - \min_{v \in [n]} \mu_v y_2(s) \right) \diff s,
\end{aligned}
\end{equation}
with $y_1(0) = y_2(0) = \int_0^\infty \bar{\phi}_0(\diff x)$. Now, we claim that
$y_1(t) \leq \int_0^\infty \bar{\phi}_t(\diff x) \leq y_2(t). $
To see why, note that $y_1(t)$, $\int_0^\infty \bar{\phi}_t(\diff x)$ and $y_2(t)$ are differentiable as a function of $t$. If $t$ is any time such that $\int_0^\infty \bar{\phi}_t(\diff x) = y_1(t)$, then
\begin{equation}
    \frac{\diff}{\diff t} \int\limits_0^\infty \bar{\phi}_t(\diff x)
    = \lambda - \int\limits_0^\infty x \bar{\phi}_t(\diff x)
    \geq \lambda - \max_{v \in [n]} \mu_v \int\limits_0^\infty \bar{\phi}_t(\diff x)
    = \lambda - \max_{v \in [n]} \mu_v y_1(t)
    = \frac{\diff}{\diff t} y_1(t),
\end{equation}
and, if $t$ is any time such that $\int_0^\infty \bar{\phi}_t(\diff x) = y_2(t)$,
\begin{equation}
    \frac{\diff}{\diff t} \int\limits_0^\infty \bar{\phi}_t(\diff x)
    = \lambda - \int\limits_0^\infty x \bar{\phi}_t(\diff x)
    \leq \lambda - \min_{v \in [n]} \mu_v \int\limits_0^\infty \bar{\phi}_t(\diff x)
    = \lambda - \max_{v \in [n]} \mu_v y_2(t)
    = \frac{\diff}{\diff t} y_2(t),
\end{equation}
from which the claim follows. Moreover, Equation \eqref{eq:def_y} implies that
\begin{equation}
\begin{aligned}
    \frac{\diff}{\diff t} \left\lvert y_1(t) - \frac{\lambda}{\max_{v \in [n]} \mu_v} \right\rvert
    &= - \left( \max_{v \in [n]} \mu_v \right) \left\lvert \frac{\lambda}{\max_{v \in [n]} \mu_v} - y_1(t) \right\rvert, \\
    \frac{\diff}{\diff t} \left\lvert y_2(t) - \frac{\lambda}{\min_{v \in [n]} \mu_v} \right\rvert
    &= - \left(\min_{v \in [n]} \mu_v \right) \left\lvert \frac{\lambda}{\min_{v \in [n]} \mu_v} - y_2(t) \right\rvert,
\end{aligned}
\end{equation}
and hence
\begin{equation}
\begin{aligned}
    \left\lvert y_1(t) - \frac{\lambda}{\max_{v \in [n]} \mu_v} \right\rvert
    = \left\lvert y_1(0) - \frac{\lambda}{\max_{v \in [n]} \mu_v} \right\rvert e^{ -\max\limits_{v \in [n]} \mu_v t}, \\
    \left\lvert y_2(t) - \frac{\lambda}{\min_{v \in [n]} \mu_v} \right\rvert
    = \left\lvert y_2(0) - \frac{\lambda}{\min_{v \in [n]} \mu_v} \right\rvert e^{-\min\limits_{v \in [n]} \mu_v t},
\end{aligned}
\end{equation}
which completes the proof.
\end{proof}




\end{document}